\documentclass[12pt,reqno]{article}

\usepackage[usenames]{color}
\usepackage{amssymb}
\usepackage{amsmath}
\usepackage{amsthm}
\usepackage{amsfonts}
\usepackage{amscd}
\usepackage{graphicx}

\usepackage[colorlinks=true,
linkcolor=webgreen,
filecolor=webbrown,
citecolor=webgreen]{hyperref}

\definecolor{webgreen}{rgb}{0,.5,0}
\definecolor{webbrown}{rgb}{.6,0,0}

\usepackage{color}
\usepackage{fullpage}
\usepackage{float}

\usepackage{graphicx}

\usepackage{latexsym}
\usepackage{epsf}
\usepackage{breakurl}

\usepackage{colortbl}
\definecolor{lightgray}{gray}{0.8}

\newcommand{\seqnum}[1]{\href{https://oeis.org/#1}{\underline{#1}}}

\begin{document}

\newtheorem*{beatty}{Theorem A} 
\newtheorem*{beattystar}{Theorem A$^*$}
\newtheorem*{uspensky}{Theorem B}
\newtheorem*{uspenskystar}{Theorem B$^*$}

\newcommand{\Ba}{B_{\alpha}}
\newcommand{\Bb}{B_{\beta}}
\newcommand{\Bc}{B_{\gamma}}
\newcommand{\Bd}{B_{\delta}}
\newcommand{\Be}{B_{\epsilon}}

\newcommand{\Bat}{\widetilde{B_{\alpha}}}
\newcommand{\Bbt}{\widetilde{B_{\beta}}}
\newcommand{\Bct}{\widetilde{B_{\gamma}}}

\newcommand{\Ebt}{E_{\beta}}
\newcommand{\Ect}{E_{\gamma}}

\newcommand{\fp}[1]{\{#1\}}
\newcommand{\Fp}[1]{\left\{#1\right\}}
\newcommand{\fl}[1]{{\lfloor #1 \rfloor}}
\newcommand{\Fl}[1]{{\left\lfloor #1 \right\rfloor}}
\newcommand{\cl}[1]{{\lceil #1 \rceil}}

\newcommand{\ZZ}{\mathbb{Z}}
\newcommand{\NN}{\mathbb{N}}
\newcommand{\QQ}{\mathbb{Q}}
\newcommand{\RR}{\mathbb{R}}

\renewcommand{\aa}{\alpha}
\newcommand{\bb}{\beta}
\newcommand{\cc}{\gamma}

\newcommand{\at}{\widetilde{a}}
\newcommand{\bt}{\widetilde{b}}
\newcommand{\ct}{\widetilde{c}}

\newcommand{\as}{\alpha^*}
\newcommand{\bs}{\beta^*}
\newcommand{\cs}{\gamma^*}


\theoremstyle{plain}
\newtheorem{theorem}{Theorem}
\newtheorem{corollary}[theorem]{Corollary}
\newtheorem{lemma}[theorem]{Lemma}
\newtheorem{proposition}[theorem]{Proposition}

\theoremstyle{definition}
\newtheorem{definition}[theorem]{Definition}
\newtheorem{example}[theorem]{Example}
\newtheorem{conjecture}[theorem]{Conjecture}

\theoremstyle{remark}
\newtheorem{remark}[theorem]{Remark}

\newcommand{\seq}[1]{(#1)_{n\in\NN}}

\begin{center}
\vskip 1cm{\LARGE\bf Almost Beatty Partitions}
\vskip 1cm 
\large
A.~J.~Hildebrand and Xiaomin Li\\
Department of Mathematics\\
University of Illinois\\
1409 W.\ Green St.\\
Urbana, IL 61801\\
USA\\
\href{mailto:ajh@illinois.edu}{\tt ajh@illinois.edu}
\\
\href{mailto:xiaomin3@illinois.edu}{\tt xiaomin3@illinois.edu}
\\
\ 
\\
Junxian Li\\
Mathematisches Institut\\
Bunsenstra\ss e 3--5\\
D-37073 G\"ottingen\\
Germany\\
\href{mailto:Junxian.Li@mathematik.uni-goettingen.de}{\tt 
Junxian.Li@mathematik.uni-goettingen.de} \\
\ \\
Yun Xie\\
Department of Statistics\\
University of Washington\\
Box 354322\\
Seattle, WA 98195\\
USA\\
\href{mailto:yunxie@uw.edu}{yunxie@uw.edu}
\end{center}

\begin{abstract}
Given $0<\alpha<1$, the Beatty sequence of density $\alpha$
is the  sequence $B_{\alpha}=(\lfloor n/\alpha\rfloor)_{n\in\NN}$.
Beatty's theorem  states that if $\alpha,\beta$ are irrational
numbers with $\alpha+\beta=1$, then the Beatty sequences
$B_{\alpha}$ and $B_{\beta}$ partition the positive integers; that
is, each positive integer belongs to exactly one of these two sequences.
On the other hand, Uspensky showed that this result breaks down completely
for partitions into three (or more) sequences: There does not exist a single
triple $(\alpha,\beta,\gamma)$ such that the Beatty sequences
$B_\alpha,B_\beta,B_\gamma$ partition the positive integers. 

In this paper we consider the question of how close we can come to a
three-part Beatty partition by considering ``almost'' Beatty sequences,
that is, sequences that represent small perturbations of an ``exact''
Beatty sequence.  We first characterize all cases in which there exists a
partition into two exact Beatty sequences and one almost Beatty sequence
with given densities, and we determine the approximation error involved. We
then give two general constructions that yield partitions into one
exact Beatty sequence and two almost Beatty sequences with prescribed
densities, and we determine the approximation error in these constructions.
Finally, we show that in many situations these constructions are
best-possible in the sense that they yield the closest approximation to a
three-part Beatty partition.

\end{abstract}

\section{Introduction}
\label{sec:introduction}

A \emph{Beatty sequence} is a sequence of the form
$\Ba=\seq{\fl{n/\alpha}}$, where
$0<\alpha<1$ is a real number\footnote{We 
are using the notation $\fl{n/\alpha}$ 
instead of the more common notation $\fl{\alpha n}$ for the $n$th term of
a Beatty sequence.  This convention has the advantage that 
the parameter $\alpha$ has a
natural interpretation as the density of the sequence $\Ba$, and it
simplifies the statements and proofs of our results.} 
and the bracket notation denotes
the floor (or greatest integer) function.  Beatty sequences arise in a
variety of areas, from diophantine approximation and dynamical systems to
theoretical computer science and the theory of quasicrystals.  They have a
number of remarkable properties, the most famous of which is the following
theorem.  Here, and in the sequel, $\NN=\{1,2,3,\dots\}$ denotes the set of
positive integers, and by a \emph{partition} of $\NN$ we mean a collection of
subsets of $\NN$ such that every element of $\NN$ belongs to exactly one of
these subsets.

\begin{beatty}[Beatty's theorem, \cite{beatty}]
Two Beatty sequences $\Ba$ and $\Bb$ partition $\NN$ if and
only if $\alpha$ and $\beta$ are positive irrational 
numbers with $\alpha+\beta=1$.
\end{beatty}

This result was posed in 1926 by Samuel Beatty as a problem in the
\emph{American Mathematical Monthly} \cite{beatty}, though it had appeared
some 30 years earlier in a book by Rayleigh on the theory of sound
\cite[p.~123]{rayleigh-book}. The result has been rediscovered
multiple times since (e.g., \cite{bang1957, skolem1957}),
and it also appeared as a problem 
in the 1959 Putnam Competition \cite{putnam1959}.  The theorem has been
variably referred to as Beatty's theorem \cite{holshouser-reiter},
Rayleigh's theorem \cite{schoenberg-book}, and the Rayleigh-Beatty theorem
\cite{obryant2002}.  For more about Beatty's theorem, we refer to the paper
by Stolarsky \cite{stolarsky1976} and the references cited therein.  Recent
papers on the topic include Ginosar and Yona \cite{ginosar-yona2012} and
Kimberling and Stolarsky \cite{kimberling-stolarsky2016}.

A notable feature of Beatty's theorem is its generality: the only 
assumptions needed to ensure that two Beatty sequences form a partition of
$\NN$ are (1) the obvious requirement that the two densities add up to
$1$, and (2) the (slightly less obvious, but not hard to verify) condition
that the numbers $\alpha$ and $\beta$ be irrational.

As effective as Beatty's theorem is in producing partitions of $\NN$ 
into \emph{two} Beatty sequences, the result breaks down completely for
partitions into \emph{three} (or more) such sequences:

\begin{uspensky}[Uspensky's theorem, \cite{uspensky1927}]
There exists no partition of $\NN$ into three or more Beatty sequences.
\end{uspensky}

This result was first proved in 1927 by Uspensky \cite{uspensky1927}.
Other proofs have been given by Skolem \cite{skolem1957} and 
Graham \cite{graham1963}.  As with Beatty's theorem, Uspensky's theorem 
eventually made its way into the Putnam, appearing as Problem
B6 in the 1995 William Lowell Putnam Competition.\footnote{Interestingly, 
neither the ``official'' solution \cite{putnam1995}
published in the \emph{Monthly}, nor the two solutions presented in the
compilation \cite{kedlaya-putnam}, while mentioning the connection to
Beatty's theorem,  made any reference to Uspensky's theorem and its  
multiple proofs in the literature. Thus, the Putnam problem may well
represent yet another rediscovery of this result.}

The theorems of Beatty and Uspensky can be stated in the following 
equivalent fashion, which highlights the complete breakdown of the partition
property when more than two Beatty sequences are involved: 

\begin{beattystar}[Complement version of Beatty's theorem]
The complement of a Beatty sequence with irrational density
is \textbf{always} another Beatty sequence.
\end{beattystar}

\begin{uspenskystar}[Complement version of Uspensky's theorem]
The complement of a union of two or more pairwise disjoint
Beatty sequences is \textbf{never} a Beatty sequence. 
\end{uspenskystar}

In light of the results of Beatty and Uspensky on the existence (resp.,
non-existence) of partitions into two (resp., three) Beatty sequences it is
natural to ask how close we can come to a partition of $\NN$ into three
Beatty sequences.  Specifically, given irrational densities
$\alpha,\beta,\gamma\in(0,1)$ that sum to $1$,
can we obtain a proper partition of $\NN$ by
slightly ``perturbing'' one or more of the Beatty sequences $\Ba,\Bb,\Bc$?
If so, how many of the sequences $\Ba,\Bb,\Bc$ do we need to perturb 
and what is the minimal amount of perturbation needed?

In this paper we seek to answer such questions by constructing partitions
into ``almost'' Beatty sequences that approximate, in an appropriate sense,
``exact'' Beatty sequences.  We call such partitions ``almost Beatty
partitions.'' 

Since, by Uspensky's theorem, partitions into three Beatty sequences do not
exist, the best we can hope for is partitions into two exact Beatty
sequences, $\Ba$ and $\Bb$, and an almost Beatty sequence
$\Bct$. In Theorem \ref{thm:thm1}  we characterize 
all triples $(\alpha,\beta,\gamma)$ of irrational numbers for which such a
partition exists, and we determine the precise ``distance'' (in the sense
of \eqref{eq:def-a-norm} below) between the almost Beatty sequence
$\Bct$ and the corresponding exact Beatty sequence $\Bc$ in such a
partition.

In Theorems \ref{thm:thm2}, \ref{thm:thm3}, and \ref{thm:thm4}
we consider partitions into one exact Beatty
sequence, $\Ba$,  and two almost Beatty sequences, $\Bbt$ and $\Bct$, with
given densities $\alpha$, $\beta$, and $\gamma$. In
Theorem \ref{thm:thm2} we give a general construction of such partitions based on
iterating the two-part Beatty partition process, and we determine the approximation
errors involved.
In Theorem \ref{thm:thm3} we present a different construction 
that requires a condition  
on the relative size of the densities $\alpha,\beta,\gamma$, but leads to a
better approximation.
In Theorem \ref{thm:thm4} we consider the special case when two of the three densities
$\alpha,\beta,\gamma$ are equal, say  $\alpha =\beta$.
We show that in this case we can always obtain an almost Beatty partition
from the Beatty sequences $\Ba$, $\Bb(=\Ba)$, and $\Bc$, by shifting \emph{all}
elements of $\Bb$ and \emph{selected} elements of $\Bc$ down by
exactly $1$.  

As special cases of the above results we recover, or improve on, some 
particular three-part almost Beatty partitions of $\NN$ that have been
mentioned in the literature; see Examples \ref{ex:thm1}, \ref{ex:thm3}, and
\ref{ex:thm4}.  We emphasize that, while these particular partitions
involve sequences with densities related to the golden ratio or other
``special'' irrational numbers, our results show that almost Beatty
partitions of similar quality exist for any triple of irrational densities
that sum up to $1$, with the quality of the approximation tied mainly to
the size of these densities, and not their arithmetic nature.

Our final result, Theorem \ref{thm:thm5}, 
is a non-existence result showing that the
almost Beatty partitions obtained through the construction of Theorem
\ref{thm:thm3} represent, in many situations,
the closest approximation to a partition
into three Beatty sequences that can be obtained by \emph{any} method. 
Specifically, Theorem \ref{thm:thm5} shows that, for ``generic'' densities
$\alpha,\beta,\gamma$ with $\alpha>1/3$,  
there does not exist an almost Beatty
partition $\NN=\Ba\cup\Bbt\cup\Bct$ in which the elements of the almost
Beatty sequences $\Bbt$ and $\Bct$ differ from the corresponding elements
of the exact Beatty sequences $\Bb$ and $\Bc$ by at most $1$.

The remainder of this paper is organized as follows: In Section
\ref{sec:results} we state our main results, Theorems
\ref{thm:thm1}--\ref{thm:thm5}, along with some examples illustrating these
results.  In Section \ref{sec:lemmas} we gather some auxiliary results,
while Sections \ref{sec:proof-thm2}--\ref{sec:proof-thm5}
contain the proofs of our main results.  We conclude in Section
\ref{sec:concluding-remarks} with some remarks on related questions, 
extensions and generalizations of our results, and directions for future
research.

\section{Statement of Results}
\label{sec:results}

\subsection{Notation and conventions}

Given a real number $x$, we let $\fl{x}$ denote its floor, defined as the
largest integer $n$ such that $n\le x$, and $\fp{x}=x-\fl{x}$ its
fractional part.

We let $\NN$ denote the set of positive integers, and we 
use capital letters $A,B,\dots$, to denote subsets of $\NN$ or,
equivalently, strictly increasing sequences of positive integers. We denote
the $n$th elements of such sequences by $a(n)$, $b(n)$, etc. 

It will be convenient to extend the definition of a sequence $\seq{a(n)}$ 
indexed by the natural numbers to a sequence indexed by
the nonnegative integers by setting $a(0)=0$.  For example, this
convention allows us to consider the ``gaps'' $a(n)-a(n-1)$ for all
$n\in\NN$, with the initial gap, $a(1)-a(0)$, having the natural
interpretation as the first element of the sequence.

Given a set $A\subset \NN$, we denote by 
\begin{equation}
\label{eq:def-counting-function}
A(n)=\#\{m\le n: m\in A\}
\end{equation}
the counting function of $A$, i.e., the number of elements of $A$ that are
$\le n$.

We measure the ``closeness'' of two sequences $\seq{a(n)}$
and $\seq{b(n)}$ by the sup-norm 
\begin{align}
\label{eq:def-a-norm}
\|a-b\| &=\sup\{|a(n)-b(n)|: n\in\NN\}.
\end{align}
Thus, $\|a-b\|<\infty$ holds if and only if one sequence can be obtained
from the other by ``perturbing'' each element by a bounded quantity. 
For integers sequences the norm $\|a-b\|$ is attained and 
represents the maximal amount by which one sequence needs to be perturbed in
order to obtain the other sequence.

\subsection{Beatty sequences and almost Beatty sequences}
Given $\alpha\in(0,1)$,
we define the \emph{Beatty sequence of density $\alpha$} as 
\begin{align}
\label{eq:def-beatty-a}
\Ba&=\seq{a(n)},\quad a(n)=\fl{n/\alpha}.
\end{align}

We call a sequence $\Bat=\seq{\at(n)}$ an  \emph{almost Beatty
sequence of density $\alpha$}
if it satisfies $\|\at-a\|<\infty$, where $a(n)=\fl{n/\alpha}$ is
the $n$th term of the Beatty sequence $\Ba$.
Thus, an almost Beatty sequence is a sequence
that can be obtained by perturbing the elements of a Beatty sequence
by a bounded amount.  Equivalently, an almost Beatty sequence
of density $\alpha$ is a sequence $\Bat$ satisfying 
\[
\Bat(N)=\#\{n\le N: n\in\Bat\}=\alpha N +O(1) \quad (N\to\infty).
\]

We use the analogous notation 
\begin{align}
\label{eq:def-beatty-b}
\Bb&=\seq{\fl{n/\beta}}=\seq{b(n)}, 
\quad 
\Bbt=\seq{\bt(n)},
\\
\label{eq:def-beatty-c}
\Bc&=\seq{\fl{n/\gamma}}=\seq{c(n)},
\quad \Bct=\seq{\ct(n)},
\end{align}
to denote Beatty sequences and almost Beatty sequences
of densities $\beta$ and $\gamma$.

\emph{In the sequel, the
notations $a(n)$, $b(n)$, and $c(n)$ will always refer to the 
elements of the Beatty sequences $\Ba,\Bb,\Bc$ 
as defined in 
\eqref{eq:def-beatty-a}, \eqref{eq:def-beatty-b},
and \eqref{eq:def-beatty-c}, 
and $\at(n)$, $\bt(n)$, and $\ct(n)$ will refer to the elements of 
almost Beatty sequences $\Bat$, $\Bbt$, and $\Bct$, of densities 
$\alpha$, $\beta$, and $\gamma$, respectively.
}

\subsection{Partitions into two exact Beatty sequences and one almost
Beatty sequence}

We assume that we are given arbitrary 
positive real numbers 
$\alpha$, $\beta$, and $\gamma$ subject only to the conditions  
\begin{equation}
\label{eq:abc}
\alpha,\beta,\gamma\in \RR^+\setminus\QQ,\quad 
\alpha+\beta+\gamma=1, 
\end{equation}
which are analogous to the conditions in Beatty's theorem (Theorem A).
Our goal is to construct a partition of $\NN$ into almost Beatty sequences
$\Bat,\Bbt,\Bct$ that are as close as possible
to the exact Beatty sequences $\Ba,\Bb,\Bc$, where ``closeness'' is
measured by the distance \eqref{eq:def-a-norm}.

By Uspensky's theorem (Theorem B), a partition of $\NN$ into three exact Beatty
sequences is impossible, so the best we can hope for is a partition into
two exact Beatty sequences $\Ba$ and $\Bb$ and one almost Beatty sequence
$\Bct$.  Our first theorem shows that such partitions do exist, it gives
necessary and sufficient conditions on the densities $\alpha,\beta,\gamma$
under which such a partition exists, and it shows exactly how close the
almost Beatty sequence $\Bct$ is to the exact Beatty sequence $\Bc$.

\begin{theorem}
[Partition into two exact Beatty sequences and one almost
Beatty sequence]
\label{thm:thm1}
Let $\alpha,\beta,\gamma$ satisfy \eqref{eq:abc}. 
Then there exists a partition $\NN=\Bat\cup\Bbt\cup\Bct$ 
with $\Bat=\Ba$, $\Bbt=\Bb$ if and only if
\begin{equation}
\label{eq:thm1-condition}
r\alpha+s\beta=1\quad \text{for some $r,s\in\NN$.}
\end{equation}
If this condition is satisfied, then $\Ba$ and $\Bb$ are disjoint 
and $\Bct=\NN\setminus(\Ba\cup \Bb)$
is an almost Beatty
sequence satisfying
\begin{align}
\label{eq:thm1-bounds}
\|\ct-c\|&
=\max\left(
\left\lfloor \frac{2-\alpha}{1-\alpha}\right\rfloor,
\left\lfloor \frac{2-\beta}{1-\beta}\right\rfloor\right),
\end{align}
where $c(n)=\fl{n/\gamma}$, resp., $\ct(n)$, 
denote the $n$th elements of $\Bc$, resp., $\Bct$.
More precisely, we have
\begin{equation}
\label{eq:thm1-errors}
0\le c(n)-\ct(n)\le \max\left(
\left\lfloor \frac{2-\alpha}{1-\alpha}\right\rfloor,
\left\lfloor \frac{2-\beta}{1-\beta}\right\rfloor\right)
\quad (n\in\NN),
\end{equation}
where the upper bound is attained for infinitely many $n$.
In particular, if $\alpha$ and $\beta$ satisfy 
\begin{equation}
\label{eq:thm1-special-case}
\max(\alpha,\beta)<1/2,
\end{equation}
then we have 
\begin{equation}
\label{eq:thm1-errors-case1}
c(n)-\ct(n)\in\{0,1,2\}
\quad (n\in\NN).
\end{equation}
\end{theorem}

The error bound \eqref{eq:thm1-errors} in this theorem is best-possible in
the strongest possible sense: There does not exist a single triple
$(\alpha,\beta,\gamma)$ of irrational densities satisfying the assumptions
of Theorem \ref{thm:thm1} for which this bound can be improved. In
particular, since 
\begin{equation*}
\left\lfloor \frac{2-t}{1-t}\right\rfloor
=\left\lfloor 1+\frac{1}{1-t}\right\rfloor\ge 2
\quad (0<t<1),
\end{equation*}
it follows that there exists no partition into two exact Beatty sequences and
an almost Beatty sequence whose elements differ from the elements of the
corresponding exact Beatty sequence by at most $1$.
Put differently, if only one of the three Beatty sequences is perturbed, the
minimal amount of perturbation (in the sense of the distance
\eqref{eq:def-a-norm}) needed in order to obtain a partition of $\NN$ is
$2$. Moreover, by Theorem \ref{thm:thm1} a perturbation by at most $2$ is
sufficient if and only if the conditions \eqref{eq:abc} and
\eqref{eq:thm1-condition} are satisfied and $\max(\alpha,\beta)<1/2$.

\begin{example}
\label{ex:thm1}
Let $\alpha=1/\Phi^3$, $\beta=1/\Phi^4$, $\gamma=1/\Phi$, 
where $\Phi=(\sqrt{5}+1)/2=1.61803\dots$ is the golden ratio. Using the relation 
$\Phi^2=\Phi+1$ one can check that $1/\Phi^3+1/\Phi^4+1/\Phi=1$ and 
$3/\Phi^3+2/\Phi^4=1$, so conditions \eqref{eq:abc} and  
\eqref{eq:thm1-condition} of Theorem \ref{thm:thm1}  hold. Moreover, since 
$\max(\alpha,\beta)<1/2$,  by the last part of the theorem 
the perturbation errors $c(n)-\ct(n)$
are in $\{0,1,2\}$. Table \ref{table:thm1-example} shows the partition
$\NN=\Ba\cup\Bb\cup\Bct$ obtained from the theorem, along with the
perturbation errors.  The two exact 
Beatty sequences $\Ba$ and $\Bb$ in this partition are 
the sequences \seqnum{A004976} and \seqnum{A004919} in OEIS \cite{oeis},
while the almost Beatty sequence $\Bct$ is a perturbation  of the sequence
\seqnum{A000201}, obtained by subtracting an appropriate amount in $\{0,
1, 2\}$ from the elements in \seqnum{A000201}.

\begin{table}[H]
\begin{center}
\begin{tabular}{|l|ccccccccccccccc|}
\hline
\rowcolor{lightgray}
$a(n)$ 
&4&8&12&16&21&25&29&33&38&42&46&50&55&59&63
\\
\hline
\rowcolor{lightgray}
$b(n)$ &6&13&20&27&34&41&47&54&61&68&75&82&89&95&102
\\
\hline
$c(n)$ &1&3&4&6&8&9&11&12&14&16&17&19&21&22&24
\\
\hline
\rowcolor{lightgray}
$\ct(n)$ & 1&2&3&5&7&9&10&11&14&15&17&18&19&22&23
\\
\hline
Error &0&1&1&1&1&0&1&1&0&1&0&1&2&0&1
\\
\hline
\end{tabular}
\caption{
An almost Beatty partition with
$\alpha=1/\Phi^3$, 
$\beta=1/\Phi^4$,
$\gamma=1/\Phi$. (The three highlighted rows form the partition.)
}
\label{table:thm1-example}
\end{center}
\end{table}
\end{example}

\subsection{Partitions into one exact Beatty sequence and two almost Beatty
sequences}

We next consider partitions into one exact Beatty sequence and two almost Beatty
sequences. In contrast to the situation in Theorem \ref{thm:thm1}, here the
two almost Beatty sequences are not uniquely determined.  We give two
constructions that lead to different partitions.  Our first approach is based on
iterating the two-part Beatty partition process and leads to the following
result. 

\begin{theorem}[Partition into one exact Beatty sequence and two almost
Beatty sequences---Construction I]
\label{thm:thm2}
Let $\alpha,\beta,\gamma$ satisfy \eqref{eq:abc} and suppose
\begin{equation}
\label{eq:thm2-condition}
\frac{\beta}{\gamma}\not\in\QQ.
\end{equation}
Let
\begin{equation}
\label{eq:thm2-construction}
\bt(n)=\Fl{\Fl{\frac{1-\alpha}{\beta}n}\frac{1}{1-\alpha}},
\quad
\ct(n)=\Fl{\Fl{\frac{1-\alpha}{\gamma}n}\frac{1}{1-\alpha}}.
\end{equation}
Then the sequences $\Ba$, $\Bbt=\seq{\bt(n)}$, and
$\Bct=\seq{\ct(n)}$
form an almost Beatty partition of $\NN$  satisfying
\begin{align}
\label{eq:thm2-bounds}
\|\bt-b\|&\le 
\left\lfloor \frac{2-\alpha}{1-\alpha}\right\rfloor,
\quad \|\ct-c\|\le \left\lfloor \frac{2-\alpha}{1-\alpha}\right\rfloor.
\end{align}
In particular, if $\alpha$ satisfies
\begin{equation}
\label{eq:thm2-special-case}
\alpha<1/2,
\end{equation}
then we have 
\begin{equation}
\label{eq:thm1-bounds2}
\|\bt-b\|\le 2,\quad \|\ct-c\|\le 2.
\end{equation}
\end{theorem}

Our second construction consists 
of starting out with two exact Beatty sequences 
$\Ba$ and $\Bb$, and then shifting those elements of $\Bb$ 
that also belong to $\Ba$ to get a  sequence $\Bbt$ that is disjoint 
from $\Ba$.

\begin{theorem}[Partition into one exact Beatty sequence and two almost
Beatty sequences---Construction II]
\label{thm:thm3}
Let $\alpha,\beta,\gamma$ satisfy \eqref{eq:abc} and suppose
\begin{equation}
\label{eq:thm3-condition}
\max(\alpha,\beta)<\gamma. 
\end{equation}
Let  $\Bbt=\seq{\bt(n)}$ be defined by 
\begin{equation}
\label{eq:thm3-bt-def}
\bt(n)=\begin{cases}
b(n),
&\text{if $b(n)\not\in\Ba$;}
\\
b(n)-1,&\text{if $b(n)\in\Ba$,}
\end{cases}
\end{equation}
and let $\Bct=\NN\setminus(\Ba\cup \Bbt)=\seq{\ct(n)}$. 
Then the sequences $\Ba$, $\Bbt$, and $\Bct$
form an almost Beatty partition of $\NN$  satisfying
\begin{align}
\label{eq:thm3-bounds}
\|\bt-b\|&\le 1,\quad \|\ct-c\|\le 2.
\end{align}
More precisely, we have 
\begin{align}
\label{eq:thm3-bounds-b}
b(n)-\bt(n)&\in\{0,1\}
\quad (n\in\NN),
\\
\label{eq:thm3-bounds-c}
c(n)-\ct(n)&\in\{0,1,2\}
\quad (n\in\NN).
\end{align}
\end{theorem}

Our proof will yield a more precise result that gives necessary and
sufficient  conditions for each of the three possible values in 
\eqref{eq:thm3-bounds-c}, and which allows one, in principle, 
to determine the relative frequencies of these values.

\begin{example}
\label{ex:thm3}
Let $\tau=1.83929\dots$ be the Tribonacci constant, defined as 
the positive root of $1/\tau+1/\tau^2+1/\tau^3=1$,
and let $\alpha=1/\tau^3$, $\beta=1/\tau^2$, $\gamma=1/\tau$.
Then $\alpha,\beta,\gamma$ satisfy the conditions \eqref{eq:abc} and 
\eqref{eq:thm3-condition} of Theorem \ref{thm:thm3}.
Thus, Theorem \ref{thm:thm3} can be applied to
yield a partition of $\NN$ into the exact Beatty sequence $\Ba=B_{1/\tau^3}$ and 
two almost Beatty sequences $\Bbt$ and $\Bct$ of densities $1/\tau^2$ and
$1/\tau$, respectively, with perturbation errors in $\{0,1\}$ for the
former sequence, and in $\{0,1,2\}$ for the latter sequence.
The resulting partition is
shown in Table \ref{table:tribonacci} below.

The sequences $\Ba,\Bb,\Bc$ are the OEIS
sequences \seqnum{A277723}, 
\seqnum{A277722},
\seqnum{A158919}, 
respectively, while sequence \seqnum{A277728}
represents the numbers not in any of these sequences.  The OEIS entry for
the latter sequence mentions three related sequences, \seqnum{A003144},
\seqnum{A003145}, and
\seqnum{A003146}, that do form a partition of $\NN$ and which differ from $\Ba$,
$\Bb$, and $\Bc$ by at most $3$. The partition obtained by Theorem
\ref{thm:thm3} is
different from this partition, and it yields a better approximation to a
proper Beatty partition, with perturbation errors of $2$ in the case of
$\Bct$, and $1$ in the case of $\Bbt$.

\begin{table}[H]
\begin{center}
\begin{tabular}{|l|c|c|c|c|c|c|c|c|c|c|c|c|c|c|c|}
\hline
\rowcolor{lightgray}
$a(n)$ & 6&12&18&24&31&37&43&49&56&62&68&74&80&87&93
\\
\hline
$b(n)$ & 3&6&10&13&16&20&23&27&30&33&37&40&43&47&50
\\
\hline
\rowcolor{lightgray}
$\bt(n)$ & 3&5&10&13&16&20&23&27&30&33&36&40&42&47&50
\\
\hline
Error & 0&1&0&0&0&0&0&0&0&0&1&0&1&0&0
\\
\hline
$c(n)$ & 1&4&6&7&10&13&13&14&17&19&21&23&24&25&28
\\
\hline
\rowcolor{lightgray}
$\ct(n)$ & 1&3&5&7&9&11&12&14&16&18&20&22&23&25&27
\\
\hline
Error & 0&1&1&0&1&2&1&0&1&1&1&1&1&0&1
\\
\hline
\end{tabular}
\caption{
An almost Beatty partition with
$\alpha=1/\tau^3$, 
$\beta=1/\tau^2$,
$\gamma=1/\tau$.
}
\label{table:tribonacci}
\end{center}
\end{table}
\end{example}

It is interesting to compare the constructions of Theorem \ref{thm:thm2}
and Theorem \ref{thm:thm3}. 
Both constructions are applicable under slightly different additional 
conditions beyond \eqref{eq:abc}: Theorem \ref{thm:thm2} requires that $\beta/\gamma$ be
irrational, while Theorem \ref{thm:thm3} requires that $\gamma>\max(\alpha,\beta)$.
Thus, the results are not directly comparable. However, in cases where both
constructions can be applied, the construction of Theorem \ref{thm:thm3} yields
stronger bounds than those that can be obtained from Theorem \ref{thm:thm2}, namely
$\|\bt-b\|\le 1$ and $\|\ct-c\|\le 2$ instead of 
$\|\bt-b\|\le 2$ and $\|\ct-c\|\le 2$.

\subsection{Partitions into one exact Beatty sequence and two 
almost Beatty sequences: The case of two equal densities}

In the case when the densities $\alpha$ and $\beta$ are equal,
we can improve on the bounds 
\eqref{eq:thm3-bounds} of Theorem \ref{thm:thm3}. 

\begin{theorem}[Partition into one exact Beatty sequence and two almost
Beatty sequences---Special case]
\label{thm:thm4}
Let $\alpha,\beta,\gamma$ satisfy \eqref{eq:abc}, and 
suppose $\beta=\alpha$. Let 
$\Bbt=\seq{\bt(n)}$ be defined by 
\begin{equation}
\label{eq:thm4-bt-def}
\bt(n)= b(n)-1\quad (n\in\NN),
\end{equation}
and let $\Bct=\NN\setminus(\Ba\cup \Bbt)=\seq{\ct(n)}$. 
Then the sequences $\Ba$, $\Bbt$, and $\Bct$
form an almost Beatty partition of $\NN$  satisfying
\begin{equation}
\label{eq:thm4-1-1-bounds}
\|\bt-b\|=1,\quad \|\ct-c\|\le 1.
\end{equation}
More precisely, we have 
\begin{align}
\label{eq:thm4-b-bound}
b(n)-\bt(n)&=1
\quad (n\in\NN), 
\\
\label{eq:thm4-c-bound}
c(n)-\ct(n)&\in\{0,1\}
\quad (n\in\NN).
\end{align}
\end{theorem}

\begin{example}
\label{ex:thm4}
Let $\alpha=\beta=1/\Phi^2$ and $\gamma=1/\Phi^3$, where $\Phi$ is the
golden ratio.  Since $1/\Phi^3=1/\Phi-1/\Phi^2$, we have
$\alpha+\beta+\gamma=1/\Phi^2+1/\Phi=1$, so the density condition
\eqref{eq:abc} is satisfied and Theorem \ref{thm:thm4} yields 
an almost Beatty partition consisting of the sequences
$\Ba=\seq{\fl{\Phi^2n}}$, 
$\Bb=\Ba-1=\seq{\fl{\Phi^2n}-1}$, 
and $\Bct$, where the elements of $\Bct$ differ from those of 
$\Bc=\seq{\fl{\Phi^3n}}$ by at most $1$.
In fact, the elementary identity 
$\fl{\Phi^2n}-1=\fl{\Phi\fl{\Phi n}}$
shows that this partition is the same as that obtained by
Theorem \ref{thm:thm2}, namely 
\[
\NN=\seq{\fl{\Phi^2n}}\cup
\seq{\fl{\Phi\fl{\Phi n}}}\cup \seq{\fl{\Phi\fl{\Phi^2 n}}}.
\]
This particular partition is known. It was mentioned in Skolem
\cite[p.~68]{skolem1957}, and it can be interpreted in terms of Wythoff
sequences; see the OEIS sequence \seqnum{A003623}. 

We remark that, while the identity $\fl{\Phi^2n}-1=\fl{\Phi\fl{\Phi n}}$,
and hence the connection with the
iterated Beatty partition construction, is closely tied to properties of 
the golden ratio, Theorem \ref{thm:thm4} shows that almost Beatty
partitions of the same quality (i.e., with one exact Beatty sequence and
two almost Beatty sequences with perturbation errors at most $1$) exist
whenever two of the densities $\alpha,\beta,\gamma$ are equal.
\end{example}

The constructions of both Theorem \ref{thm:thm4}
and Theorem \ref{thm:thm1} can be viewed as special
cases of the construction of Theorem \ref{thm:thm3}. Indeed, if $\Ba=\Bb$, then the
formula for $\bt(n)$ of Theorem \ref{thm:thm3}  reduces to $\bt(n)=b(n)-1$ for all
$n$,  while in the case when $\Ba$ and $\Bb$ are disjoint, this formula
yields $\bt(n)=b(n)$ for all $n$ and thus $\Bbt=\Bb$.  We note, however, 
that Theorems \ref{thm:thm1}  and \ref{thm:thm4}
are more general in one respect: they
do not require the condition \eqref{eq:thm3-condition} of Theorem
\ref{thm:thm3}.

\subsection{A non-existence result}
As mentioned in the remarks following Theorem \ref{thm:thm1}, the bounds on
the perturbation errors in this result are best-possible. As a consequence,
there exists no partition $\NN=\Bat\cup\Bbt\cup\Bct$  such that $\Bat=\Ba$
and $\Bbt=\Bb$ are exact Beatty sequences and $\Bct$ is an almost Beatty
sequence satisfying $\|\ct-c\|\le 1$.

A similarly universal optimality result does not hold for the bound 
$\|c-\ct\|\le 2$ in Theorem \ref{thm:thm3}. Indeed, as Theorem
\ref{thm:thm4} shows, in the case of two equal densities this bound can be
improved to $\|c-\ct\|\le 1$. 
In the following theorem we show that, for ``generic'' densities
$\alpha,\beta,\gamma$ with $\alpha>1/3$,  the error 
bound $\|c-\ct\|\le 2$ is indeed best-possible. 

\begin{theorem}
[Non-existence of partitions into an exact Beatty sequence and two
almost Beatty sequences with perturbation errors $\le 1$] 
\label{thm:thm5}
Let $\alpha,\beta,\gamma$ satisfy \eqref{eq:abc} and suppose that 
\begin{align}
\label{eq:thm5-alpha-bound}
&\alpha>1/3,
\\
&\text{$1,\alpha,\beta$ are linearly independent over $\QQ$.}
\label{eq:thm5-linear-independence}
\end{align}
There exists no partition $\NN=\Bat\cup\Bbt\cup\Bct$  such that
$\Bat=\Ba$ is an exact Beatty sequence
and $\Bbt=\seq{\bt(n)}$ and $\Bct=\seq{\ct(n)}$ 
are almost Beatty sequences of densities
$\beta$ and $\gamma$, respectively, satisfying  $\|\bt-b\|\le 1$ and
$\|\ct-c\|\le 1$.
\end{theorem}

\section{Lemmas}
\label{sec:lemmas}

We begin by stating, without proof, some elementary relations involving 
the floor and fractional part functions.

\begin{lemma}[Floor and fractional part function identities]
\label{lem:floor-identities}
For any real numbers $x,y$ we have
\begin{align*}
\fl{x+y}&= \fl{x} + \fl{y}  + \delta(\fp{x},\fp{y}),
\\
\fl{x-y}&= \fl{x} - \fl{y}  - \delta(\fp{x-y},\fp{y}),
\\
\fp{x+y}&= \fp{x} + \fp{y}  - \delta(\fp{x},\fp{y}),
\\
\fp{x-y}&= \fp{x} - \fp{y}  + \delta(\fp{x-y},\fp{y}),
\end{align*}
where 
\begin{equation}
\label{eq:delta-def}
\delta(s,t)=\begin{cases} 
1, &\text{if $s+t\ge 1$;}
\\
0, &\text{if $s+t< 1$.}
\end{cases}
\end{equation}
\end{lemma}

In the following lemma we collect some elementary properties of Beatty
sequences.  We will provide proofs for the sake of completeness.  

\begin{lemma}[Elementary properties of Beatty sequences]
\label{lem:beatty-properties}
Let $\alpha\in(0,1)$ be irrational, and let 
$\Ba=\seq{a(n)}=\seq{\fl{n/\alpha}}$ be the Beatty sequence of density $\alpha$. 
\begin{itemize}
\item[(i)] \textbf{Membership criterion:}
For any $m\in\NN$ we have
\begin{equation*}
m\in\Ba
\Longleftrightarrow \fp{(m+1)\alpha}<\alpha
\Longleftrightarrow \fp{m\alpha}>1-\alpha.
\end{equation*}
\item[(ii)] \textbf{Counting function formula:} 
For any $m\in\NN$ we have
\begin{equation*}
\Ba(m)=\fl{\alpha (m+1)}=\alpha m+\alpha-\fp{\alpha (m+1)},
\end{equation*}
where $\Ba(m)$ is the counting function of $\Ba$, as defined in 
\eqref{eq:def-counting-function}.

\item[(iii)]  \textbf{Gap formula:}
Let $k=\fl{1/\alpha}$.
For any $n\in\NN$ we have 
\begin{equation*}
a(n+1)-a(n)=
\begin{cases}
k+1,&\text{if $\fp{n/\alpha}\ge 1-\fp{1/\alpha}$;}
\\
k,&\text{otherwise.}
\end{cases}
\end{equation*}

\item[(iv)]  \textbf{Gap criterion:}
Given $m\in \Ba$, let
$m'$ denote the successor to $m$ in the sequence 
$\Ba$, so that, by (iii), $m'=m+k$ or $m'=m+k+1$,
where $k=\fl{1/\alpha}$.
Then we have, for any $m\in\NN$,
\begin{align*}
m\in\Ba\text{ and } m'=m+k
&\Longleftrightarrow 
\fp{1/\alpha}\alpha< \fp{(m+1)\alpha}<\alpha,
\\
m\in\Ba\text{ and } m'=m+k+1
&\Longleftrightarrow 
\fp{(m+1)\alpha}<\fp{1/\alpha}\alpha.
\end{align*}

\end{itemize}
\end{lemma}

\begin{remark} 
Our assumption that $\alpha$ is irrational ensures that
equality cannot hold in any of the above relations.
\end{remark}

\begin{proof}
(i) We have 
\begin{align*}
m\in\Ba&\Longleftrightarrow
m=\fl{n/\alpha}\text{ for some $n\in\NN$}
\\
&\Longleftrightarrow  m< n/\alpha<m+1\text{ for some $n\in\NN$}
\\
&\Longleftrightarrow  m\alpha < n<(m+1)\alpha
\text{ for some $n\in\NN$}
\\
&\Longleftrightarrow \fp{(m+1)\alpha}<\alpha
\\
&\Longleftrightarrow \fp{m\alpha}>1-\alpha.
\end{align*}

(ii) We have
\begin{align*}
\Ba(m)=n&\Longleftrightarrow
\fl{n/\alpha}\le m\le \fl{(n+1)/\alpha}-1
\\
&\Longleftrightarrow
n/\alpha< m+1< (n+1)/\alpha
\\
&\Longleftrightarrow
n< (m+1)\alpha<  n+1
\\
&\Longleftrightarrow 
n=\fl{(m+1)\alpha} = (m+1)\alpha-\fp{(m+1)\alpha}.
\end{align*}

(iii)
By Lemma \ref{lem:floor-identities} we have
\[
a(n+1)-a(n)=\Fl{\frac{n+1}{\alpha}}-\Fl{\frac{n}{\alpha}}
=\Fl{\frac1{\alpha}}+
\delta\left(\Fp{\frac{1}{\alpha}},\Fp{\frac{n}{\alpha}}\right),
\]
where the last term
is equal to $1$ if $\fp{1/\alpha}+\fp{n/\alpha}\ge 1$,
and $0$ otherwise.

(iv)
Note that
\[
k\alpha=\left(\frac{1}{\alpha}-\Fp{\frac{1}{\alpha}}\right)\alpha
=1-\Fp{\frac1{\alpha}}\alpha.
\]
Thus, using the results of part (i) and (iii) we have
\begin{align*}
m\in\Ba\text{ and } m'=m+k&\Longleftrightarrow
m\in\Ba\text{ and } m+k\in\Ba
\\
&\Longleftrightarrow
\fp{(m+1)\alpha}<\alpha\text{ and } 
\fp{(m+1)\alpha+k\alpha}<\alpha
\\
&\Longleftrightarrow
\fp{(m+1)\alpha}<\alpha\text{ and } 
\fp{(m+1)\alpha-\fp{1/\alpha}\alpha}<\alpha
\\
&\Longleftrightarrow
\fp{1/\alpha}\alpha<\fp{(m+1)\alpha}<\alpha.
\end{align*}
This proves the first of the asserted equivalences in (iv).  The second
equivalence,  asserting that $m\in\Ba$ and $m'=m+k+1$ holds if and only if 
$\fp{(m+1)\alpha}<\fp{1/\alpha}\alpha$, follows from this on observing
that, by (i), $m\in\Ba$ is equivalent to
$\fp{(m+1)\alpha}\in(0,\alpha)$, and by (iii) for any $m\in\Ba$, we have
either $m'=m+k$ or $m'=m+k+1$.
\end{proof}

For our next lemma we assume that $\alpha,\beta,\gamma$ are
irrational numbers satisfying \eqref{eq:abc}, so that, in particular, 
\begin{equation}
\label{eq:abc-gamma}
\gamma=1-\alpha-\beta,
\end{equation}
and we define
\begin{equation}
\label{eq:um-vm}
u_m=\{(m+1)\alpha\}, \quad 
v_m=\{(m+1)\beta\}, \quad 
w_m=\{(m+1)\gamma\}.
\end{equation}
Using \eqref{eq:abc-gamma}, the latter quantity, $w_m$, 
can be expressed in terms of $u_m$ and $v_m$:
\begin{align}
\label{eq:wm}
w_m &=\fp{(m+1)(1-\alpha-\beta)}=\fp{-(m+1)(\alpha+\beta)}
\\
\notag
&=1-\fp{(m+1)\alpha+(m+1)\beta}
=1-\fp{u_m+v_m}.
\end{align}

\begin{lemma}[Counting function identity]
\label{lem:Ba-Bb-Bc-counting-function}
Let $\alpha,\beta,\gamma$ satisfy \eqref{eq:abc}, and let $u_m$ and 
$v_m$ be as in \eqref{eq:um-vm}.
Then we have, for any $m\in\NN$,  
\begin{equation}
\label{eq:abc-sum}
\Ba(m)+\Bb(m)+\Bc(m)=m-\delta(u_m,v_m),
\end{equation}
where $\delta(s,t)\in\{0,1\}$ is given by \eqref{eq:delta-def}.
\end{lemma}

\begin{proof}
Using Lemma \ref{lem:beatty-properties}(ii) together with
\eqref{eq:abc-gamma} and \eqref{eq:wm} we obtain
\begin{align*}
\Ba(m)+\Bb(m)+\Bc(m)&=(m+1)(\alpha+\beta+\gamma)- u_m-v_m-w_m
\\
&=m -u_m-v_m+\fp{u_m+v_m}.
\end{align*}
Since $u_m,v_m\in (0,1)$, Lemma \ref{lem:floor-identities} yields 
$\fp{u_m+v_m}=u_m+v_m-\delta(u_m,v_m)$.
The desired identity now follows.
\end{proof}

The remaining lemmas in this section are deeper results which we quote from
the literature. The first of these results characterizes disjoint Beatty
sequences; see Theorem 3.11 in Niven \cite{niven-book}.

\begin{lemma}[Disjointness criterion]
\label{lem:disjointness}
Let $\alpha,\beta$ be irrational numbers in $(0,1)$. Then the
Beatty sequences $\Ba$ and $\Bb$ are disjoint if and only
if there exist positive integers $r$ and $s$ such that
\begin{equation}
\label{eq:disjointness}
 r\alpha+s\beta=1.
\end{equation}
 \end{lemma}

The next lemma is a special case of Weyl's Theorem in the theory of uniform
distribution modulo $1$; see Examples 2.1 and 6.1 in Chapter 1 of Kuipers
and Niederreiter \cite{kuipers-niederreiter}.

\begin{lemma}[Weyl's Theorem]
\label{lem:weyl}
\mbox{}
\begin{itemize}
\item[(i)] Let $\theta$ be an irrational number. Then the sequence 
$\seq{n\theta}$ is \emph{uniformly distributed modulo $1$}; that is,  
we have
\[
\lim_{N\to\infty}
\frac1N\#\{n\le N: \fp{n\theta}<t\}=t\quad (0\le t\le 1).
\]
\item[(ii)] Let $\theta_1,\dots,\theta_k$ be real numbers such that 
the numbers $1,\theta_1,\dots,\theta_k$ are linearly independent over
$\QQ$. Then the $k$-dimensional sequence $\{(n\theta_1,\dots,n\theta_k)\}$
is \emph{uniformly distributed modulo $1$ in $\RR^k$}; that is,  
we have
\[
\lim_{N\to\infty}
\frac1N\#\{n\le N: \fp{n\theta_i}<t_i\text{ for $i=1,\dots,k$}\}
=t_1\dots t_k\quad (0\le t_i\le 1).
\]
\end{itemize}
\end{lemma}

\section{Proof of Theorem \protect\ref{thm:thm2}}
\label{sec:proof-thm2}

Let $\alpha,\beta,\gamma$ satisfy \eqref{eq:abc} and
\eqref{eq:thm2-condition}. Thus, $\alpha,\beta,\gamma$ are positive 
irrational numbers with $\alpha+\beta+\gamma=1$ and $\beta/\gamma$ 
irrational.

Since $\alpha$ is irrational, we can apply Beatty's theorem
(Theorem A) to the pair of densities 
$(\alpha,1-\alpha)$ to obtain a partition
$\NN=\Ba\cup B_{1-\alpha}$, where  $\Ba$ is the exact Beatty
sequence in the three-part partition we are trying to construct and
\begin{equation*}
B_{1-\alpha}=\left(\Fl{n\frac{1}{1-\alpha}}\right)_{n\in\NN}.
\end{equation*}
We partition the latter sequence
by partitioning the index set $\NN$ into
two Beatty sequences with densities $\beta/(1-\alpha)$ 
and $1-\beta/(1-\alpha)=\gamma/(1-\alpha)$.
These densities sum to $1$, and our assumption that $\beta/\gamma$ is
irrational ensures that both densities are irrational.
Thus Beatty's theorem can be applied again, yielding the partition 
\begin{align*}
B_{1-\alpha}=\left(\Fl{\Fl{\frac{1-\alpha}{\beta}n}\frac{1}{1-\alpha}}
\right)_{n\in\NN}
\cup \left(\Fl{\Fl{\frac{1-\alpha}{\gamma}n}
\frac{1}{1-\alpha}}\right)_{n\in\NN}.
\end{align*}
The two sequences in this partition
are exactly the sequences $\Bbt=\seq{\bt(n)}$
and $\Bct=\seq{\ct(n)}$ defined in Theorem \ref{thm:thm2}. Thus, it remains to
show that these sequences satisfy the bounds \eqref{eq:thm2-bounds}.
By symmetry, it suffices to prove the first of these bounds, $\|\bt-b\|\le 
\fl{(2-\alpha)/(1-\alpha)}$.

Using elementary properties of the floor and fractional part functions 
we have, for all $n\in\NN$,
\begin{align*}
\bt(n)&=\Fl{\Fl{\frac{1-\alpha}{\beta}n}\frac{1}{1-\alpha}}
\le \Fl{\frac{n}{\beta}}=b(n)
\end{align*}
and
\begin{align*}
\bt(n)&
=\Fl{\frac{n}{\beta}-\Fp{\frac{1-\alpha}{\beta}n}\frac{1}{1-\alpha}}
\ge\Fl{\frac{n}{\beta}}-\Fl{\frac{1}{1-\alpha}}-1
\\
&=b(n) -\Fl{\frac{1}{1-\alpha}}-1.
\end{align*}
Hence 
\begin{equation*}
0\le b(n)-\bt(n)\le
\Fl{\frac{1}{1-\alpha}}+1=
\Fl{\frac{2-\alpha}{1-\alpha}}
\end{equation*}
for all $n\in\NN$. It follows that $\|\bt-b\|\le
\fl{(2-\alpha)/(1-\alpha)}$, which is the asserted bound
\eqref{eq:thm2-bounds} for $\|\bt-b\|$. 
This completes the proof of Theorem \ref{thm:thm2}.

\section{Proof of Theorem \protect\ref{thm:thm3}}
\label{sec:proof-thm3}

Let $\alpha,\beta,\gamma$ satisfy the conditions \eqref{eq:abc} and
\eqref{eq:thm3-condition} of Theorem \ref{thm:thm3}.
Thus, $\alpha,\beta,\gamma$ are
positive irrational numbers satisfying $\alpha+\beta+\gamma=1$ and
$\gamma>\max(\alpha,\beta)$. The latter two conditions imply 
\begin{equation}
\label{eq:thm3proof-density-condition}
\alpha<1/2,\quad \beta<1/2,\quad \gamma>1/3.
\end{equation}
Let $\Bbt=\{\bt(n)\}$
and $\Bct=\NN\setminus(\Ba\cup\Bbt)$ be the sequences defined in
Theorem \ref{thm:thm3}.

\subsection{Proof of the partition property}
We first show that the sequences
$\Bbt$ and $\Ba$ are disjoint.  Consider an element $m=\bt(n)\in\Bbt$. By
definition, we have $m=b(n)$ if $b(n)\not\in\Ba$, and $m=b(n)-1$ if
$b(n)\in\Ba$.  In the former case, we immediately get $m\not\in\Ba$, while
in the latter case we have $m+1\in\Ba$, which by Lemma
\ref{lem:beatty-properties}(iii) and the above assumption $\alpha<1/2$  
implies $m\not\in\Ba$. 
Thus the sequences $\Ba$ and $\Bbt$ are disjoint.

Since $\Bct$ is defined as the complement of the sequences $\Ba$ and
$\Bbt$, it follows that the three
sequences $\Ba$, $\Bbt$, $\Bct$ form a partition of $\NN$, as claimed.
\qed

\subsection{Proof of the bounds
\eqref{eq:thm3-bounds} and
\eqref{eq:thm3-bounds-c}}
The norm estimates \eqref{eq:thm3-bounds} obviously follow from the 
definition of $\bt(n)$ and \eqref{eq:thm3-bounds-c}, so it suffices 
to prove the latter relation, i.e., 
\begin{equation}
\label{eq:thm3proof-bounds-c}
c(n)-\ct(n)\in\{0,1,2\}\quad\text{for all $n\in\NN$.}
\end{equation}

We break up the argument into several lemmas. 
We recall the notation 
\begin{align}
\label{eq:recall-um-vm}
u_m&=\{(m+1)\alpha\}, \quad v_m=\{(m+1)\beta\}, 
\quad w_m=\{(m+1)\gamma\}=1-\fp{u_m+v_m}.
\end{align}

We set  
\begin{equation}
\label{eq:def-Egamma}
\Ebt(m)=\Bbt(m)-\Bb(m),\quad 
\Ect(m)=\Bct(m)-\Bc(m).
\end{equation}
Thus, the sequences $\Ebt$ and $\Ect$ represent the differences between
the counting functions of the almost Beatty sequences $\Bbt$ and $\Bct$
and the counting functions of the corresponding exact Beatty
sequences.  In the following two lemmas we show that these differences
are always in $\{0,1\}$, and we characterize the cases in which each of the
values $0$ and $1$ is
taken on.

\begin{lemma}
\label{lem:Ebt-formula}
We have, for all $m\in\NN$,
\begin{equation*}
\Ebt(m)= 
\begin{cases}
1,&\text{if $u_m> 1-\alpha$ and $v_m> 1-\beta$;}
\\
0, &\text{otherwise.}
\end{cases}
\end{equation*}
\end{lemma}

\begin{proof}
By the definition of the sequence $\Bbt=\seq{\bt(n)}$
we have $\bt(n)=b(n)-1$
if $b(n)\in\Ba\cap\Bb$, and $\bt(n)=b(n)$ otherwise.  In the first case, we
necessarily have $b(n)+1\not\in\Bb$ and thus $b(n)\not\in\Bbt$ since, by
Lemma \ref{lem:beatty-properties}(iii) and our assumption
\eqref{eq:thm3proof-density-condition}, the difference between consecutive
elements of $\Bb$ is at least $\ge \fl{1/\beta}\ge 2$.
It follows that the counting functions of $\Bbt$ and $\Bb$ satisfy 
\[
\Bbt(m)=\begin{cases}
\Bb(m)+1,&\text{if $m+1\in\Ba\cap\Bb$;}
\\
\Bb(m),&\text{otherwise.}
\end{cases}
\]
By Lemma \ref{lem:beatty-properties}(i), $m+1\in\Ba\cap\Bb$ holds if and
only if $\fp{(m+1)\alpha}>1-\alpha$ and $\fp{(m+1)\beta}>1-\beta$,
which, by the definition of the numbers $u_m$ and $v_m$, is equivalent 
to the condition $u_m>1-\alpha$ and $v_m>1-\beta$.  
The assertion of the lemma then follows on noting that 
$\Ebt(m)=\Bbt(m)-\Bb(m)$.
\end{proof}

\begin{lemma}
\label{lem:Ect-formula}
We have, for all $m\in\NN$,
\begin{equation*}
\Ect(m)=
\begin{cases}
1, &\text{if $u_m+v_m> 1\text{ and } 
\bigl(u_m<1-\alpha\text{ or } v_m<1-\beta\bigr)$;}
\\
0, &\text{otherwise.}
\end{cases}
\end{equation*}
\end{lemma}

\begin{proof}
Using the relation $\Bct=\NN\setminus(\Ba\cup\Bbt)$ 
along with Lemma \ref{lem:Ebt-formula}, we get, for all $m\in\NN$.
\[
\Bct(m)=m-\Ba(m)-\Bbt(m)=m-\Ba(m)-\Bb(m)-\eta(u_m,v_m),
\]
where $\eta(u,v)=1$ if $u>1-\alpha$ and $v>1-\beta$, and $\eta(u,v)=0$
otherwise. On the other hand, 
Lemma \ref{lem:Ba-Bb-Bc-counting-function} yields 
\[
\Bc(m)=m-\Ba(m)-\Bb(m)-\delta(u_m,v_m),
\]
where $\delta(u,v)$ is given by \eqref{eq:delta-def}, i.e., 
$\delta(u,v)=1$ if $u+v\ge 1$, and $\delta(u,v)=0$ otherwise.
Hence, 
\[
\Ect(m)=\Bct(m)-\Bc(m)=\delta(u_m,v_m)-\eta(u_m,v_m).
\]
It follows that $\Ect(m)=1$ if and only if $\delta(u_m,v_m)=1$ and
$\eta(u_m,v_m)=0$, i.e., if and only if $u_m+v_m\ge1$ 
and $u_m< 1-\alpha$ or $v_m< 1-\beta$.   The latter conditions are
exactly the conditions in the lemma characterizing the case $\Ect(m)=1$. 

If these conditions are not satisfied, then we have either
(i) $\delta(u_m,v_m)=\eta(u_m,v_m)$, and hence $\Ect(m)=0$, or 
(ii) $\delta(u_m,v_m)=0$ and $\eta(u_m,v_m)=1$, in which case we would have 
$\Ect(m)=-1$. Thus, to complete the proof, it suffices to show that 
case (ii) is impossible.  Indeed, 
the conditions $\delta(u_m,v_m)=0$ and $\eta(u_m,v_m)=1$
are equivalent to the three inequalities 
$u_m+v_m<1$, $u_m>1-\alpha$, and $v_m> 1-\beta$ holding simultaneously. 
But the latter two inequalities imply $u_m+v_m>2-\alpha-\beta$, which
contradicts the first inequality since,
by \eqref{eq:thm3proof-density-condition},
$\alpha+\beta<1$.
\end{proof}

\begin{lemma}
\label{lem:thm3-Bct}
We have, for all $m\in\NN$,  
\begin{align*}
m\in\Bct&\Longleftrightarrow 
u_m>\alpha\text{ and } v_m>\beta\text{ and } 
\bigl( u_m<1-\alpha\text{ or } v_m<1-\beta\bigr).
\end{align*}
\end{lemma}

\begin{proof}
By construction, we have $\Bct=\NN\setminus(\Ba\cup \Bbt)$, 
with
\[
\Bbt=\{m\in\NN: m\in \Bb\setminus\Ba \text{ or }m+1\in\Bb\cap \Ba \}.
\]
Therefore, using Lemma \ref{lem:beatty-properties}(i), we get 
\begin{align*}
m\in\Bct&\Longleftrightarrow m\not\in \Ba\text{ and } m\not\in \Bbt
\\
&\Longleftrightarrow m\not\in \Ba\text{ and } 
m\not\in \Bb\setminus\Ba\text{ and } m+1\not\in\Bb\cap\Ba
\\
&\Longleftrightarrow u_m>\alpha \text{ and } 
v_m>\beta \text{ and } 
\bigl( u_m<1-\alpha\text{ or } v_m<1-\beta\bigr).
\qedhere
\end{align*}
\end{proof}

\begin{lemma}
\label{lem:thm3-cn-1-ctn-cn}
We have, for  all $n\in\NN$,
\begin{equation*}
c(n-1)\le \ct(n)\le c(n).
\end{equation*}
\end{lemma}

\begin{proof}
Note that $\Bct(\ct(n))=n$ and $\Bct(\ct(n)-1)=n-1$.  
Thus, if $\ct(n)\ge c(n)+1$, then, using 
the monotonicity of the counting function $\Bct(m)$ and 
Lemma \ref{lem:Ect-formula}, we have 
\begin{align*}
n-1&=\Bct(\ct(n)-1))\ge \Bct(c(n))\\
&=\Bc(c(n))+\Ect(c(n))
\\
&=n+\Ect(c(n))\ge n,
\end{align*}
which is a contradiction.  This proves the desired upper bound,
$\ct(n)\le c(n)$.

Similarly, if $\ct(n)\le c(n-1)-1$, then, 
using the relation $\Bc(c(n-1)-1)=n-2$, we get
\begin{align*}
n&=\Bct(\ct(n))\le \Bct(c(n-1)-1)
\\
&=\Bc(c(n-1)-1)+
\Ect(c(n-1)-1)
\\
&=n-2+ \Ect(c(n-1)-1)\le n-1,
\end{align*}
which is again a contradiction, thus proving
the lower bound $\ct(n)\ge c(n-1)$. 
\end{proof}

\begin{lemma}
\label{lem:thm3-ct-c}
We have, for all $n\in\NN$, 
\begin{equation*}
c(n)-2\le \ct(n)\le c(n).
\end{equation*}
\end{lemma}

\begin{proof} 
Note that $\gamma>1/3$ by our
assumption \eqref{eq:thm3proof-density-condition}. 
Hence, Lemma \ref{lem:beatty-properties}(iii) yields 
\[
c(n-1)\ge c(n)-\fl{1/\gamma}-1\ge c(n)-3,
\]
and combining this with Lemma \ref{lem:thm3-cn-1-ctn-cn} we obtain
\begin{equation}
\label{eq:thm3proof-cn-3a}
c(n)-3\le c(n-1)\le \ct(n)\le c(n).
\end{equation}
If $\gamma>1/2$, then the same argument yields the stronger bound 
\[
c(n)-2\le c(n-1)\le \ct(n)\le c(n),
\]
which proves the asserted inequality.
Thus, we may assume 
\begin{equation}
\label{eq:thm3proof-gamma-bound}
\frac13<\gamma<\frac12,
\end{equation}
and it remains to show that in this case we cannot have $c(n)-3=\ct(n)$. 

We argue by contradiction. Suppose that $c(n)-3=\ct(n)$. In view of 
\eqref{eq:thm3proof-cn-3a}, this forces the double equality
\begin{equation}
\label{eq:thm3proof-cn-3}
\ct(n)=c(n-1)=c(n)-3.
\end{equation}
Our assumption \eqref{eq:thm3proof-gamma-bound} then implies
$3=\fl{1/\gamma}+1$. Thus,
$3$ is the larger of the two possible gaps in the
Beatty sequence $\Bc$, and by Lemma \ref{lem:beatty-properties}(iv) 
it follows that $\fp{(m+1)\gamma}<\fp{1/\gamma}\gamma$, where
$m=c(n-1)$.  Since (see \eqref{eq:um-vm} and \eqref{eq:wm})
\[
\fp{(m+1)\gamma}=w_m=1-\fp{u_m+v_m}, 
\]
the latter condition is equivalent to 
\begin{equation}
\label{eq:thm3proof-cn-3b}
\fp{u_m+v_m}>1-\fp{1/\gamma}\gamma.
\end{equation}

On the other hand, our assumption \eqref{eq:thm3proof-cn-3} implies
$\Bc(m)=\Bc(c(n-1))=n-1$, $\Bct(m)=\Bct(\ct(n))=n$, and hence $\Ect(m)=1$.
By Lemma \ref{lem:Ect-formula} the latter condition holds if and only if 
$u_m+v_m>1$  and  at least one of the inequalities 
$u_m<1-\alpha$ and $v_m<1-\beta$ holds. 
It follows that $\fp{u_m+v_m}=u_m+v_m-1$ and hence
\begin{equation}
\label{eq:thm3proof-cn-3c}
\fp{u_m+v_m}=u_m+v_m-1<1-\min(\alpha,\beta).
\end{equation}

Comparing \eqref{eq:thm3proof-cn-3b}
with \eqref{eq:thm3proof-cn-3c} yields
\[
\min(\alpha,\beta)<\fp{1/\gamma}\gamma=1-2\gamma,
\]
where the last step follows since, 
by our assumption $1/3<\gamma<1/2$, $\fp{1/\gamma}=(1/\gamma)-2$.
But the latter condition implies $2\gamma+\min(\alpha,\beta)<1$, 
which contradicts the assumptions $\alpha+\beta+\gamma=1$ and
$\gamma>\max(\alpha,\beta)$ of Theorem \ref{thm:thm3}.
Hence, the case $c(n)-3=\ct(n)$ is
impossible.

This completes the proof of Lemma \ref{lem:thm3-ct-c}.
\end{proof}

\subsection{Distribution of perturbation errors}
Lemma \ref{lem:thm3-ct-c} implies the asserted relation
\eqref{eq:thm3proof-bounds-c} and thus completes the proof of 
Theorem \ref{thm:thm3}.
In the remainder of this section we study the distribution 
of the  perturbation errors $c(n)-\ct(n)$ more closely.

\begin{lemma}
\label{lem:thm3-perturbation-errors}
Given $n\in\NN$, let $m=\ct(n)$ (so that, in particular, $m\in\Bct$). Then
\begin{equation*}
c(n)-\ct(n)=
\begin{cases}
0&\Longleftrightarrow m\in\Bct \text{ and } \Ect(m)=0;
\\
1&\Longleftrightarrow
\Ect(m)=1\text{ and } m\in\Bct \text{ and } m+1\in\Bc;
\\
2&\Longleftrightarrow
\Ect(m)=1\text{ and } m\in\Bct \text{ and } m+1\not\in\Bc.
\end{cases}
\end{equation*}
\end{lemma}

\begin{proof}
Let $m=\ct(n)$ and $m'=c(n)$, so that $\Bct(m)=n$, $\Bc(m')=n$.
 By Lemma \ref{lem:thm3-ct-c}
we have $m'\in\{m,m+1,m+2\}$. 

If $m'=m$, then $\Bct(m)=n=\Bc(m)$, so
$\Ect(m)=\Bct(m)-\Bc(m)=0$. On the other hand, 
if $m'\ge m+1$, then 
\[
\Bc(m)\le \Bc(m'-1)=\Bc(c(n)-1)=n-1=\Bct(m)-1
\]
and hence $\Ect(m)=1$. 
Thus, $m'=m$ holds if and only if $\Ect(m)=0$, while 
$m'\in\{m+1,m+2\}$ holds if and only if $\Ect(m)=1$.
In the latter case, if $m+1\in\Bc$, then $m'=m+1$, 
while if $m+1\not\in\Bc$, then $m'>m+1$, and thus necessarily $m'=m+2$.  
This yields the desired characterization of the values 
$c(n)-\ct(n)$ (i.e., $m'-m$).
\end{proof}

\begin{proposition}[Characterization of perturbation errors]
\label{prop:thm3-perturbation-errors}
Under the assumptions of Theorem \ref{thm:thm3} we have, for any $n\in\NN$, with
$m=\ct(n)$,  
\begin{equation*}
c(n)-\ct(n)=
\begin{cases}
0&\Longleftrightarrow 
u_m>\alpha \text{ and } v_m> \beta \text{ and } u_m+v_m<1;
\\
1&\Longleftrightarrow
u_m>\alpha \text{ and } v_m> \beta \text{ and } 1<u_m+v_m<1+\gamma;
\\
2&\Longleftrightarrow
\bigl(u_m<1-\alpha\text{ or }v_m<1-\beta\bigr) \text{ and }
u_m+v_m>1+\gamma.
\end{cases}
\end{equation*}
\end{proposition}

\begin{proof}
In view of Lemma \ref{lem:thm3-perturbation-errors} it suffices to 
show that the three sets of conditions in this lemma are equivalent to the
corresponding conditions in the proposition. 

By Lemmas \ref{lem:Ect-formula} and \ref{lem:thm3-Bct} we have
\begin{align}
\label{eq:Bct-cond}
m\in\Bct&\Longleftrightarrow 
u_m>\alpha\text{ and } v_m>\beta\text{ and } 
\bigl( u_m<1-\alpha\text{ or } v_m<1-\beta\bigr),
\\
\label{eq:Ect1-cond}
\Ect(m)=1
&\Longleftrightarrow
u_m+v_m>1\text{ and } \bigl(u_m<1-\alpha\text{ or } v_m<1-\beta\bigr),
\\
\label{eq:Ect0-cond}
\Ect(m)=0
&\Longleftrightarrow
u_m+v_m<1\text{ or } \bigl(u_m>1-\alpha\text{ and } v_m>1-\beta\bigr).
\end{align}
Hence the first condition in Lemma \ref{lem:thm3-perturbation-errors},
``$m\in\Bct$ and $\Ect(m)=0$'',
holds if and only if 
the three conditions $u_m>\alpha$, $v_m>\beta$, and $u_m+v_m<1$, hold
simultaneously, i.e., if and only if the first condition in the proposition
is satisfied.  This proves the first of the three asserted equivalences.

For the second equivalence, note that, by Lemma
\ref{lem:beatty-properties}(i),  \eqref{eq:um-vm},  and \eqref{eq:wm},  
\begin{align}
\label{eq:Bc-cond}
m+1\in\Bc
&\Longleftrightarrow w_m>1-\gamma\Longleftrightarrow\fp{u_m+v_m}<\gamma
\\
\notag
&\Longleftrightarrow
u_m+v_m<\gamma\text{ or }1<u_m+v_m<1+\gamma.
\end{align}
Combining this with the conditions
\eqref{eq:Bct-cond} and \eqref{eq:Ect1-cond}
for $m\in\Bct$ and $\Ect(m)=1$, 
we see that the second condition in Lemma 
\ref{lem:thm3-perturbation-errors}, i.e., 
``$m\in\Bct$ and $\Ect(m)=1$ and $m+1\in\Bc$'', 
holds if and only if
\begin{equation}
\label{eq:lem5.6cond2-equivalence}
1<u_m+v_m<1+\gamma \text{ and } 
u_m> \alpha\text{ and }
v_m>\beta\text{ and }
(u_m<1-\alpha\text{ or } v_m<1-\beta).
\end{equation}
Now note that the first condition in \eqref{eq:lem5.6cond2-equivalence}
implies $u_m+v_m< 2-\alpha-\beta=(1-\alpha)+(1-\beta)$, 
so that at least one of  $u_m<1-\alpha$
and $v_m<1-\beta$ must hold.  Thus the last condition 
in \eqref{eq:lem5.6cond2-equivalence} is a consequence of the first condition
and can therefore be dropped.  Hence \eqref{eq:lem5.6cond2-equivalence}, and
therefore also the second condition in Lemma
\ref{lem:thm3-perturbation-errors}, is equivalent to the second condition in
the proposition.

The equivalence between the third conditions in Lemma
\ref{lem:thm3-perturbation-errors} and the proposition can be seen by a
similar argument.
\end{proof}

The conditions of Proposition \ref{prop:thm3-perturbation-errors}
can be described geometrically as
follows:
\begin{equation*}
c(n)-\ct(n)=d\Longleftrightarrow (u_m,v_m)\in R_d\quad (d=0,1,2),
\end{equation*}
where $m=\ct(n)$ and  $R_0$, $R_1$, and $R_2$ are the regions inside the
unit square shown in Figure \ref{fig:thm3} below.

\begin{figure}[H]
\begin{center}
\hspace{.05\textwidth}\includegraphics[width=0.9\textwidth]{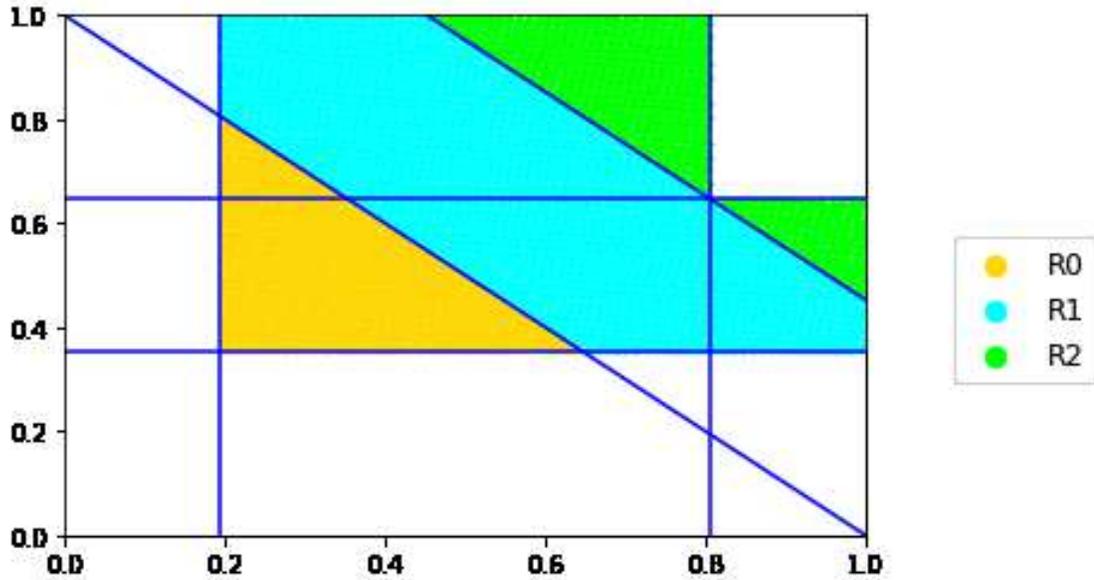}
\end{center}
\caption{The regions $R_0$, $R_1$, and $R_2$ in the $uv$-plane 
representing the conditions given in Proposition 
\ref{prop:thm3-perturbation-errors}
for the perturbation errors $c(n)-\ct(n)$ to be 
equal to $0$, $1$, and $2$, respectively.
The vertical boundary lines 
are given by the equations $u=\alpha$ and $u=1-\alpha$,  
the horizontal boundary lines 
are given by the equations $v=\beta$,  $v=1-\beta$,
and the two diagonal lines are given by $u+v=1$ and $u+v=1+\gamma$.
}
\label{fig:thm3}
\end{figure}

If $\alpha,\beta,1$ are linearly independent over $\QQ$, then
Weyl's Theorem (Lemma \ref{lem:weyl}) implies that the points
$(u_m,v_m)$ are uniformly distributed in the unit square.
By Lemma \ref{lem:thm3-Bct} and Proposition
\ref{prop:thm3-perturbation-errors}
it follows that, under this linear independence
condition, the densities of the perturbation errors 
\begin{equation}
\label{eq:density-thm3}
P(d)=\lim_{N\to\infty}\frac1N\#\{n\le N: c(n)-\ct(n)=d\}
\end{equation}
are given by $P(d)=A(R_d)/A(R_0\cup R_1\cup R_2)$,
where $A(R)$ denotes the area of a region $R$. 
An elementary computation of the areas of the regions $R_d$
depicted in Figure \ref{fig:thm3} then yields 
\begin{align*}
P(0)&=\frac{\gamma}{2},\quad
P(1)=1-\frac{\alpha^2+\beta^2+\gamma^2}{2\gamma},
\quad
P(2)=\frac{\alpha^2+\beta^2}{2\gamma}.
\end{align*}
In particular, under the density conditions 
\eqref{eq:thm3proof-density-condition} of Theorem \ref{thm:thm3} and the above linear
independence assumption, it follows all three densities are nonzero.
Hence, under these conditions the perturbation errors $c(n)-\ct(n)$ take on
each of the values $0,1,2$ on a set of positive density.

\section{Proof of Theorem \protect\ref{thm:thm4}}
\label{sec:proof-thm4}

Let $\alpha$, $\beta$ and $\gamma$ be as in Theorem \ref{thm:thm4}. Thus
$\beta=\alpha$ and
$\alpha$ and $\gamma$ are positive irrational numbers satisfying
\begin{equation}
\label{eq:thm4proof-gamma-condition}
2\alpha+\gamma=1.
\end{equation}
In particular, we must have
\begin{equation}
\label{eq:thm4proof-density-condition}
\alpha<\frac12.
\end{equation}

Let $\Bbt=\seq{\bt(n)}$, where $\bt(n)=b(n)-1$, and
$\Bct=\NN\setminus(\Ba\cup\Bbt)=\seq{\ct(n)}$ 
be the sequences defined in Theorem \ref{thm:thm4}. 
Using \eqref{eq:thm4proof-density-condition}, we see as in the proof of
Theorem \ref{thm:thm3} that the sequences $\Ba$, $\Bbt$, and $\Bct$ form a partition 
of $\NN$. The relations $b(n)-\bt(n)=1$ and $\|\bt-b\|=1$ follow 
directly from the
definition of $\bt(n)$.
Thus, to complete the proof of Theorem \ref{thm:thm4}
it remains to prove that 
\begin{equation}
\label{eq:thm4main}
c(n)-\ct(n)\in\{0,1\}.
\end{equation}
This will follow from the three lemmas below.

We recall the notations  (see \eqref{eq:recall-um-vm} and
\eqref{eq:def-Egamma}) 
\[
\Ect(m)=\Bct(m)-\Bc(m)
\]
and 
\[
u_m=\{(m+1)\alpha\}, \quad v_m=\{(m+1)\beta\}, 
\quad w_m=1-\fp{u_m+v_m}
\]
from the proof of Theorem \ref{thm:thm3}. 
We note that, by our assumption $\alpha=\beta$, we have 
\begin{equation}
\label{eq:um=vm}
u_m=v_m, \quad w_m=1-\fp{2u_m}.
\end{equation}

The first two lemmas are the special cases $\alpha=\beta$ (and thus $u_m=v_m$)
of Lemma \ref{lem:Ect-formula} and Lemma \ref{lem:thm3-cn-1-ctn-cn}, 
respectively\footnote{Theorem \ref{thm:thm3} involved the additional condition
$\gamma>\max(\alpha,\beta)$, which in the case when $\alpha=\beta$  
reduces to $\gamma>\alpha$. However, the proofs of Lemmas  
\ref{lem:Ect-formula} and \ref{lem:thm3-cn-1-ctn-cn}, while requiring the
upper bounds $\alpha<1/2$ and $\beta<1/2$,  did not make use of  
this lower bound for $\gamma$.  Thus the conclusions of these lemmas remain
valid under the present conditions, i.e., $\alpha=\beta<1/2$.}.

\begin{lemma}
\label{lem:thm4-Ect-formula}
We have, for all $n\in\NN$,
\begin{equation*}
\Ect(m)=
\begin{cases}
1, &\text{if $1/2<u_m<1-\alpha$;}
\\
0, &\text{otherwise.}
\end{cases}
\end{equation*}
\end{lemma}

\begin{lemma}
\label{lem:thm4-cn-1-ctn-cn}
We have, for all $n\in\NN$,
\begin{equation*}
c(n-1)\le \ct(n)\le c(n).
\end{equation*}
\end{lemma}

The following lemma is a stronger version of Lemma
\ref{lem:thm3-cn-1-ctn-cn}, valid under the conditions of Theorem
\ref{thm:thm4}.

\begin{lemma}
\label{lem:thm4-ct-c}
We have, for all $n\in\NN$,
\begin{equation*}
c(n)-1\le \ct(n)\le c(n),
\end{equation*}
\end{lemma}

\begin{proof} 
The upper bound, $\ct(n)\le c(n)$, follows from
Lemma \ref{lem:thm4-cn-1-ctn-cn}.  For the lower bound, 
suppose $\ct(n)\le c(n)-2$. 
Let $m=\ct(n)$.  Then 
\[
\Bc(m)\le\Bc(m+1) \le \Bc(c(n)-1)=n-1,
\]
while 
\[
\Bct(m+1)\ge \Bct(m)=\Bct(\ct(n))=n.
\]
Thus we have 
\[
\Ect(m)=\Bct(m)-\Bc(m)\ge 1\text{ and }
\Ect(m+1)=\Bct(m+1)-\Bc(m+1)\ge1. 
\]
By Lemma \ref{lem:thm4-Ect-formula} 
this implies that 
\begin{equation*}
1/2<u_m<1-\alpha\text{ and }1/2<\fp{u_m+\alpha}<1-\alpha.
\end{equation*}
The latter two relations imply 
\begin{equation}
\label{eq:thm4proof-1}
1/2<u_m<1-2\alpha,
\end{equation}
and hence, by 
\eqref{eq:thm4proof-gamma-condition}, 
$\gamma=1-2\alpha>1/2$. 
Therefore $\fl{1/\gamma}=1$, and 
by Lemma \ref{lem:beatty-properties}(iii)
it follows that 
$c(n)-c(n-1)\in\{1,2\}$.  Our assumption $\ct(n)\le c(n)-2$ and  
Lemma \ref{lem:thm4-cn-1-ctn-cn} then implies 
\[
c(n-1)=\ct(n)=c(n)-2.
\]
By Lemma \ref{lem:beatty-properties}(iv) 
it follows that, with $m=c(n-1)=\ct(n)$, we have
\begin{equation}
\label{eq:thm4proof-2}
w_m<\fp{1/\gamma}\gamma=((1/\gamma)-1)\gamma=1-\gamma=2\alpha.
\end{equation}
On the other hand, using \eqref{eq:um=vm}
and \eqref{eq:thm4proof-1} we have
\begin{equation}
\label{eq:thm4proof-3}
w_m=1-\fp{2u_m} =1-(2u_m-1)=2(1-u_m)>2(2\alpha).
\end{equation}
Comparing \eqref{eq:thm4proof-2}
and \eqref{eq:thm4proof-3} yields $4\alpha<2\alpha$, which is a contradiction.
Hence we must have
$c(n)-\ct(n)\in\{0,1\}$. This completes the proof of Lemma
\ref{lem:thm4-ct-c} and of Theorem \ref{thm:thm4}.
\end{proof}

\section{Proof of Theorem \protect\ref{thm:thm1}}
\label{sec:proof-thm1}

The necessity of the condition \eqref{eq:thm1-condition} follows from
the disjointness criterion of Lemma \ref{lem:disjointness}. 
Thus, it remains to show that if this condition is satisfied, then $
\Bct=\NN\setminus(\Ba\cup\Bb)$ is an almost Beatty sequence satisfying
the bounds \eqref{eq:thm1-errors}, and that the upper bound here is
attained for infinitely many $n$.

Suppose that $\alpha,\beta,\gamma$ are irrational numbers 
satisfying the conditions \eqref{eq:abc} and \eqref{eq:thm1-condition}
of Theorem \ref{thm:thm1}.  Thus, there exist $r,s\in\NN$ such that 
\begin{equation}
\label{eq:thm1-proof-sa+sb}
r\alpha+s\beta=1.
\end{equation}
We distinguish two cases:
(I) $r=1$ or $s=1$, and (II) $r>1$ and $s>1$.

\subsection{Proof of the error bounds \eqref{eq:thm1-errors},
Case I: $r=1$ or $s=1$.} 
By symmetry it suffices to consider the case $r=1$. 
Then \eqref{eq:thm1-proof-sa+sb} reduces to $\alpha+s\beta=1$, and since
$\gamma=1-\alpha-\beta>0$, we necessarily have $s\ge 2$ and 
\begin{equation}
\label{eq:thm1-proof-gamma}
\gamma=(s-1)\beta.
\end{equation}
Moreover, since $\beta=(1-\alpha)/s<1/2$, we have $\fl{(2-\beta)/(1-\beta)}=2$
and therefore
\[
\max\left(
\left\lfloor \frac{2-\alpha}{1-\alpha}\right\rfloor,
\left\lfloor \frac{2-\beta}{1-\beta}\right\rfloor\right)
=\max\left(
\left\lfloor 1+\frac{1}{1-\alpha}\right\rfloor,2\right)
=1+\left\lfloor \frac{1}{1-\alpha}\right\rfloor.
\]
Thus, to prove the desired bounds \eqref{eq:thm1-errors}, it suffices to
show that, for all $n\in\NN$,  
\begin{equation}
\label{eq:thm1-proof-errors-case1}
0\le c(n)-\ct(n)\le 
\left\lfloor \frac{1}{1-\alpha}\right\rfloor+1.
\end{equation}

Let
\begin{equation}
\label{eq:thm1-proof-Bsb}
B_{s\beta}=\seq{\fl{n/(s\beta)}}
\end{equation}
be the Beatty sequence of density $s\beta$. Since $\alpha+s\beta =1$,
by Beatty's theorem the sequences $\Ba$ and $B_{s\beta}$ partition $\NN$.

Now observe that the sequence $\Bb=\seq{\fl{n/\beta}}$  is the
subsequence of $B_{s\beta}$ obtained by restricting the index $n$ 
in \eqref{eq:thm1-proof-Bsb} to integers $n\equiv 0\bmod s$.
Since $\Ba$ and $B_{s\beta}$ partition $\NN$, it follows that 
\begin{align*}
\Bct&=\NN\setminus(\Ba\cup \Bb)=
B_{s\beta}\setminus \Bb
=\{\fl{n/(s\beta)}: n\in\NN,n\not\equiv0\bmod s\}.
\end{align*}
Hence $\ct(m)$, the $m$th term of the sequence $\Bct$, 
is equal to $\fl{n_m/(s\beta)}$, where $n_m$ is the $m$th positive integer  
$n$ satisfying $n\not\equiv 0\bmod s$.   A simple enumeration argument
shows that if we represent $m$ (uniquely) in the form 
\begin{equation*}
m=i(s-1)+j,\quad i\in\{0,1,\dots\},\quad j\in\{1,2\dots,s-1\},
\end{equation*}
then $n_m=is+j$.
With this notation, the $m$th term of $\Bct$ is given by 
\begin{align}
\label{eq:thm1-proof-ct1}
\ct(m)=\ct(i(s-1)+j)&=\Fl{\frac{n_m}{s\beta}}=\Fl{\frac{is+j}{s\beta}}
=\Fl{\frac{i}{\beta}+\frac{j}{s\beta}}.
\end{align}
On the other hand, since, by 
\eqref{eq:thm1-proof-gamma}, $c(n)=\fl{n/\gamma}=\fl{n/((s-1)\beta)}$, we have 
\begin{align}
\label{eq:thm1-proof-c1}
c(m)&=\Fl{\frac{m}{(s-1)\beta}}=\Fl{\frac{i(s-1)+j}{(s-1)\beta}}
=\Fl{\frac{i}{\beta}+\frac{j}{(s-1)\beta}}.
\end{align}
From \eqref{eq:thm1-proof-ct1} and \eqref{eq:thm1-proof-c1}
we obtain, on using Lemma \ref{lem:floor-identities}, 
\begin{align}
\label{eq:thm1-proof-ct-c}
0&\le c(m)-\ct(m)=
\Fl{\frac{i}{\beta}+\frac{j}{(s-1)\beta}}
-\Fl{\frac{i}{\beta}+\frac{j}{s\beta}}
\\
\notag
&=\Fl{\left(\frac{i}{\beta}+\frac{j}{(s-1)\beta}\right)-
\left(\frac{i}{\beta}+\frac{j}{s\beta}\right)} +\delta_{i,j}
\\
\notag
&\le \Fl{\frac{j}{s(s-1)\beta}}+1
\le \Fl{\frac{1}{s\beta}}+1
=\Fl{\frac{1}{1-\alpha}}+1,
\end{align}
where
\begin{align*}
\delta_{i,j}&=
\delta\left(\Fp{\frac{is+j}{s\beta}},\Fp{\frac{j}{s(s-1)\beta}}\right)
\\
&=\begin{cases}
1, &\text{if  
$\Fp{\frac{is+j}{s\beta}}+\Fp{\frac{j}{s(s-1)\beta}}\ge 1$;}
\\
0, &\text{otherwise.}
\end{cases}
\end{align*}
This proves the bounds \eqref{eq:thm1-proof-errors-case1}.

\subsection{Proof of the error bounds \eqref{eq:thm1-errors},
Case II: $r>1$ and $s>1$.} 
In this case, we have $\gamma=1-\alpha-\beta=(r-1)\alpha+(s-1)\beta\ge
\alpha+\beta$, so in particular 
\begin{equation}
\label{eq:thm1-proof-gamma-bound}
\gamma>1/2,\quad \alpha<1/2,\quad \beta<1/2. 
\end{equation}
Hence 
\[
\max\left(
\left\lfloor \frac{2-\alpha}{1-\alpha}\right\rfloor,
\left\lfloor \frac{2-\beta}{1-\beta}\right\rfloor\right)
=2,
\]
so the desired bounds \eqref{eq:thm1-errors} reduce to 
\begin{equation}
\label{eq:thm1-proof-errors-case2}
0\le c(n)-\ct(n)\le 2.
\end{equation}

Now note that, by \eqref{eq:thm1-proof-gamma-bound}, we  
have $\gamma>\max(\alpha,\beta)$, so the conditions of Theorem \ref{thm:thm3}
are satisfied.  Moreover, since $\Ba$ and $\Bb$ are disjoint, the sequence
$\Bbt$ defined in this theorem is identical to $\Bb$, and consequently the
sequences $\Bct$ defined in Theorems \ref{thm:thm1} and \ref{thm:thm3}
must also be equal.  Hence, all results established in the proof of Theorem
\ref{thm:thm3}
can be applied in the current situation.
In particular, Lemma \ref{lem:thm3-ct-c} yields the desired bounds
\eqref{eq:thm1-proof-errors-case2} for the perturbation errors 
$c(n)-\ct(n)$.

\subsection{Optimality of the error bounds \eqref{eq:thm1-errors}}

To complete the proof of Theorem \ref{thm:thm1}, it remains to show that
the upper bounds in \eqref{eq:thm1-errors} are sharp.

Consider first Case I above, i.e., the case when $r=1$.  
In this case \eqref{eq:thm1-errors}  reduces to
\eqref{eq:thm1-proof-errors-case1}, and we need to show  that the upper
bound in the latter inequality is attained infinitely often.

Consider integers $m\in\NN$ satisfying $m\equiv0\bmod(s-1)$.
For such integers we have $j=s-1$ in the representation $m=i(s-1)+j$. 
Thus \eqref{eq:thm1-proof-ct-c} reduces to 
\begin{equation}
\label{eq:thm1-proof-optimality}
c(m)-\ct(m)=\Fl{\frac{1}{s\beta}}
+ \delta_{i,s-1}
=\Fl{\frac{1}{1-\alpha}}
+ \delta_{i,s-1},
\end{equation}
where 
\begin{equation*}
\delta_{i,j}
=\begin{cases}
1, &\text{if  $\Fp{\frac{is+s-1}{s\beta}}\ge 1-\Fp{\frac{1}{s\beta}}$;}
\\
0, &\text{otherwise.}
\end{cases}
\end{equation*}
Note that 
$\fp{(is+(s-1))/(s\beta)}=\fp{i\beta+\theta}$ with $\theta=(s-1)/(s\beta)$.  
Since $\beta$ is irrational, by Weyl's theorem (Lemma \ref{lem:weyl})
the sequence $(\fp{i\beta})_{i\in\NN}$,
and therefore also the shifted sequence $(\fp{i\beta+\theta})_{i\in\NN}$,
is dense in $[0,1]$. Hence $\delta_{i,s-1}=1$ 
holds for infinitely many values of $i$.
By \eqref{eq:thm1-proof-optimality}
it follows that  
$c(m)-\ct(m)= \fl{1/(1-\alpha)}+1$ 
also holds for infinitely many $m$. 
This proves that the upper bound in \eqref{eq:thm1-proof-errors-case1} is
attained for infinitely many $m$.

In Case II the bounds \eqref{eq:thm1-errors} reduce to $0\le c(n)-\ct(n)\le
2$, and we need to show that $c(n)-\ct(n)=2$ holds for infinitely many $n$.
By Proposition \ref{prop:thm3-perturbation-errors} (which, as noted above, 
is applicable under the assumptions of Case II),
$c(n)-\ct(n)=2$ holds for infinitely many $n$ if and only if the conditions
\begin{equation}
\label{eq:thm1-proof-optimality-case2}
\bigl(u_m<1-\alpha\text{ or }v_m<1-\beta\bigr) \text{ and }
u_m+v_m>1+\gamma,
\end{equation}
where $u_m=\fp{(m+1)\alpha}$ and $v_m=\fp{(m+1)\beta}$, hold for infinitely
many $m$.  Thus, it remains to show the latter assertion.

We may assume without loss of generality that $s>r$. 
(Note that the case $s=r$ is impossible since then 
\eqref{eq:abc} and \eqref{eq:thm1-condition} imply $\alpha+\beta=1/r$ and
$\gamma=1-(\alpha+\beta)=1-1/r$, contradicting the irrationality of
$\gamma$.)  Under this assumption we have
\begin{align*}
\beta&<\frac1s(r\alpha+s\beta)=\frac1s,
\\
\gamma&=1-\alpha-\beta <1-\frac1s(r\alpha+s\beta)=1-\frac1s.
\end{align*}
In view of these bounds, a sufficient condition for 
\eqref{eq:thm1-proof-optimality-case2} is
\begin{equation}
\label{eq:thm1-proof-um-vm-bounds}
u_m\in(1-\epsilon_1,1),\quad v_m\in \left(1-\frac1s,
1-\frac1s+\epsilon_2\right),
\end{equation}
provided $\epsilon_1,\epsilon_2>0$ are small enough. 
Indeed, if $\beta<(1/s)-\epsilon_2$
and $\gamma< 1-(1/s)-\epsilon_1$, then 
the upper bound for $v_m$ in 
\eqref{eq:thm1-proof-um-vm-bounds} implies $v_m<1-\beta$, while the lower
bounds for $u_m$ and $v_m$ imply $u_m+v_m>2-(1/s)-\epsilon_1>1+\gamma$.
Thus it remains to show that, given arbitrarily small $\epsilon_1$ and
$\epsilon_2$,  there exist infinitely many $m$ satisfying   
\eqref{eq:thm1-proof-um-vm-bounds}.

By \eqref{eq:thm1-condition} we have $\beta=(1-r\alpha)/s$ and hence
\begin{equation}
\label{eq:thm1-proof-vm-formula}
v_m=\fp{\beta(m+1)}=\Fp{\frac{m+1}{s}-r\frac{\alpha(m+1)}{s}}.
\end{equation}
Setting
\begin{equation*}
\label{eq:thm1-proof-umstar}
u_m^*=\Fp{\frac{\alpha(m+1)}{s}}
\end{equation*}
and letting $j\in\{1,2,\dots,s\}$ be defined by
\begin{equation*}
\label{eq:thm1-proof-m+1}
m+1\equiv j\bmod s,
\end{equation*}
we can write 
\eqref{eq:thm1-proof-vm-formula} as 
\begin{align}
\label{eq:thm1-proof-vm-formula2}
v_m&=
\Fp{\frac{j}{s}-r\frac{\alpha(m+1)}{s}}
=\Fp{\frac{j}{s}-r\Fp{\frac{\alpha(m+1)}{s}}}
\\
\notag
&
=\Fp{\frac{j}{s}-ru_m^*}=1-\Fp{ru_m^*-\frac{j}{s}}.
\end{align}

Let $0<\epsilon<1$ be given, and suppose that $m$ satisfies 
\begin{equation}
\label{eq:thm1-proof-m-condition} 
m+1\equiv r-1\bmod s
\end{equation}
and
\begin{equation}
\label{eq:thm1-proof-umstar-condition} 
u_m^*\in\left(\frac{1-\epsilon}{s},\frac{1}{s}\right).
\end{equation}
Then
\begin{align}
\label{eq:thm1-proof-um-interval} 
u_m&=
\Fp{s\frac{\alpha(m+1)}{s}}=
\Fp{s\Fp{\frac{\alpha(m+1)}{s}}}=
\fp{su_m*}\in (1-\epsilon,1).
\end{align}
Moreover, if $\epsilon<1/r$, then
\eqref{eq:thm1-proof-vm-formula2} with $j=r-1$ implies
\begin{align}
\label{eq:thm1-proof-vm-interval} 
v_m&=
1-\Fp{r u_m^*-\frac{r-1}{s}}=
1-\Fp{\frac1s+r\left(u_m^*-\frac1s\right)}
\in \left(1-\frac1s,1-\frac1s+\frac{r\epsilon}{s}\right).
\end{align}
From \eqref{eq:thm1-proof-um-interval} and 
\eqref{eq:thm1-proof-vm-interval} we see that  
the \eqref{eq:thm1-proof-um-vm-bounds}
holds with $\epsilon_1$ and $\epsilon_2$ given by 
\begin{equation}
\label{eq:thm1-proof-epsilon}
\epsilon_1=\epsilon,\quad \epsilon_2=\frac{r\epsilon}{s}.
\end{equation}
Hence, if $\epsilon$ is chosen small enough, then for any $m$ satisfying the
conditions
\eqref{eq:thm1-proof-m-condition} and 
\eqref{eq:thm1-proof-umstar-condition},
the desired bounds \eqref{eq:thm1-proof-optimality-case2} hold.

It remains to show that there are infinitely many
$m$ satisfying these two conditions.  The first of these conditions, 
\eqref{eq:thm1-proof-m-condition}, restricts $m$ to 
integers of the form $m=is+r-2$, $i=0,1,2,\dots$. For such $m$ we have
\[
u_m^*=\Fp{\frac{\alpha(m+1)}{s}}
=\Fp{\frac{\alpha(is+r-1)}{s}} =\Fp{\alpha i + \frac{(r-1)\alpha}{s}}.
\]
Since $\alpha$ is irrational, by Weyl's theorem (Lemma \ref{lem:weyl}),
the sequence $(\fp{\alpha i})_{i\ge0}$, 
and hence also the sequence $(\fp{\alpha i+((r-1)\alpha/s)})_{i\ge0}
=(u^*_{is+r-2})_{i\ge0}$,
is dense in $[0,1]$. 
Thus there exist infinitely many $m$
for which $u_m^*$ falls into the interval
$((1-\epsilon)/s,1/s)$, i.e., satisfies \eqref{eq:thm1-proof-umstar-condition}.
This proves our claim and completes the proof of 
Theorem \ref{thm:thm1}.

\begin{remark}
\label{rem:error-densities-thm1}
The above argument yields the \emph{maximal} value taken on by the
perturbation errors $c(n)-\ct(n)$.  One can ask, more generally,  
for the \emph{exact} set of values taken on by these errors, and for
the densities with which these values occur.  This is a more difficult
problem that leads to some very delicate questions on the behavior of the pairs
$(u_m,v_m)$.  We plan to address this question in a future paper. We note 
here only that, in general, it is \emph{not} the case that the
perturbation errors $c(n)-\ct(n)$ are uniformly distributed over their range
of possible values.
\end{remark}

\section{Proof of Theorem \ref{thm:thm5}}
\label{sec:proof-thm5}

Let $\alpha,\beta,\gamma$ satisfy the conditions of the theorem; that is, 
assume that $\alpha,\beta,\gamma$ are positive irrational numbers summing
to $1$ and suppose in addition that $\alpha>1/3$ and that the numbers
$1,\alpha,\beta$ are linearly independent over $\QQ$. 

We will show that, under these assumptions, 
there exist infinitely many $m\in\NN$ such that
\begin{equation}
\label{eq:thm5-proof-goal}
m\in \Ba,\quad m+1\in\Bb\cap \Bc, \quad m+2\in \Ba.
\end{equation}
The desired conclusion follows from 
\eqref{eq:thm5-proof-goal}.
Indeed, for any $m$ satisfying 
\eqref{eq:thm5-proof-goal} the element $m+1$ belongs to both $\Bb$ and
$\Bc$. Thus, in order to obtain a partition of the desired type 
(i.e., of the form $\NN=\Ba\cup \Bbt\cup \Bct$),
for one of these latter two
sequences the element $m+1$ would have to ``perturbed'' to avoid the
overlap.  However, since $m$ and $m+2$ are elements of $\Ba$,
a perturbation of $m+1$ by $\pm1$ creates an overlap with an element in
$\Ba$, in contradiction to the partition property.
Thus there does not exist a partition $\NN=\Ba\cup \Bbt\cup \Bct$
in which the perturbation errors for the almost Beatty sequences $\Bbt$ and
$\Bct$ are bounded by $1$ in absolute value.

It remains to show that \eqref{eq:thm5-proof-goal} holds for
infinitely many $m$. We distinguish two cases, $1/3<\alpha<1/2$ 
and $1/2<\alpha<1$.

Suppose first that 
\begin{equation}
\label{eq:thm5-proof-case1}
1/3<\alpha<1/2.
\end{equation}
Without loss of generality, we may assume that 
$\beta\le \gamma$, which, by \eqref{eq:thm5-proof-case1},  implies
\begin{equation}
\label{eq:thm5-proof-beta-bound}
\beta\le \frac12(\beta+\gamma)=\frac12(1-\alpha)<\frac13<\alpha.
\end{equation}

Applying part (i) of Lemma \ref{lem:beatty-properties} we obtain 
\begin{equation}
\label{eq:thm5-proof-case1-1}
m+1\in\Bb\cap \Bc\Longleftrightarrow \fp{(m+1)\beta}>1-\beta\text{ and } 
\fp{(m+1)\gamma}>1-\gamma.
\end{equation}
Using the notation (cf. \eqref{eq:um-vm} and \eqref{eq:wm})
\begin{align*}
u_m&=\{(m+1)\alpha\}, \ v_m=\{(m+1)\beta\}, 
\ w_m=\{(m+1)\gamma\}=1-\fp{u_m+v_m},
\end{align*}
and the relation $\gamma=1-\alpha-\beta$, 
the two conditions on the right of \eqref{eq:thm5-proof-case1-1}
are seen to be equivalent to 
\begin{align}
\label{eq:thm5-proof-Bb-condition}
v_m&>1-\beta,
\\
\label{eq:thm5-proof-Bc-condition}
\fp{u_m+v_m}&<1-\alpha-\beta,
\end{align}
respectively. 

Next, by part (iv) of Lemma \ref{lem:beatty-properties} 
along with our assumption \eqref{eq:thm5-proof-case1}
(which implies that the number $k$ in  Lemma
\ref{lem:beatty-properties}(iv) 
is equal to $\fl{1/\alpha}=2$) we have 
\begin{equation}
\label{eq:thm5-proof-case1-2}
m\in\Ba\text{ and }m+2\in \Ba
\Longleftrightarrow 
\fp{1/\alpha}\alpha< u_m< \alpha.
\end{equation}
Since 
$\fp{1/\alpha}\alpha =((1/\alpha)-2)\alpha=1-2\alpha$,
the condition on the right of \eqref{eq:thm5-proof-case1-2}
is equivalent to 
\begin{equation}
\label{eq:thm5-proof-Ba-condition}
1-2\alpha<u_m<\alpha.
\end{equation}

Thus $m$ satisfies \eqref{eq:thm5-proof-goal} if and only if 
the numbers $u_m$ and $v_m$ satisfy 
\eqref{eq:thm5-proof-Bb-condition},
\eqref{eq:thm5-proof-Bc-condition}, and 
\eqref{eq:thm5-proof-Ba-condition}, and 
it remains to show that the latter three conditions hold for
infinitely many $m$.  
To this end, let $\epsilon>0$ be given and  consider the intervals
\[
I_\epsilon=(\alpha-\epsilon,\alpha),\quad
J_\epsilon=(1-\beta,1-\beta+\epsilon).
\]
If $\epsilon$ is sufficiently small, then, since
$1-2\alpha<\alpha$ by \eqref{eq:thm5-proof-case1}, we have
\begin{equation}
\label{eq:thm5-proof-I-J-conditions1}
I_\epsilon\subset (1-2\alpha,\alpha), 
\quad J_\epsilon\subset (1-\beta,1).
\end{equation}
Moreover, our assumption \eqref{eq:thm5-proof-beta-bound} implies
$\alpha-\epsilon+(1-\beta)>1$ 
provided $\epsilon$ is sufficiently small, 
and hence  
\begin{equation}
\label{eq:thm5-proof-I-J-conditions2}
I_\epsilon+J_\epsilon
\subset(\alpha+1-\beta-\epsilon,\alpha+1-\beta+\epsilon)
\subset (1,2),
\end{equation}
where $I+J$ denotes the sumset of $I$ and $J$, i.e., the set of
all elements $x+y$ with $x\in I$ and $y\in J$.

Assume now that $\epsilon$ is small enough so that
\eqref{eq:thm5-proof-I-J-conditions1} and
\eqref{eq:thm5-proof-I-J-conditions2} both hold.
Let $(u_m,v_m)$ be such that 
\begin{equation}
\label{eq:thm5-proof-um-vm-condition}
u_m\in I_\epsilon\text{ and }v_m\in J_\epsilon.
\end{equation}
Then, by \eqref{eq:thm5-proof-I-J-conditions1},
\eqref{eq:thm5-proof-Ba-condition} and \eqref{eq:thm5-proof-Bb-condition}
are satisfied.
Moreover, in view of \eqref{eq:thm5-proof-I-J-conditions2},
we have 
\begin{equation}
\label{eq:thm5-proof-um+vm-condition}
\fp{u_m+v_m}=u_m+v_m-1<\alpha
+(1-\beta +\epsilon)-1=\alpha-\beta+\epsilon,
\end{equation}
and since $\alpha<1/2$, we have $\alpha+\epsilon<1/2<1-\alpha$ provided 
$\epsilon<(1/2)-\alpha$. 
Hence, for sufficiently small $\epsilon$,  the right-hand side of 
\eqref{eq:thm5-proof-um+vm-condition} is smaller than $1-\alpha-\beta$,
and condition \eqref{eq:thm5-proof-Bc-condition} is satisfied as well.

Thus, assuming $\epsilon$ is sufficiently small, 
any pair $(u_m,v_m)$ for which  
\eqref{eq:thm5-proof-um-vm-condition} holds satisfies the three conditions 
\eqref{eq:thm5-proof-Bb-condition},
\eqref{eq:thm5-proof-Bc-condition}, and 
\eqref{eq:thm5-proof-Ba-condition}.
Since $(u_m,v_m)=(\fp{(m+1)\alpha}, \fp{(m+1)\beta})$ and, by 
assumption, 
the numbers $1,\alpha,\beta$ are linearly independent over $\QQ$, 
the two-dimensional version of Weyl's Theorem (Lemma \ref{lem:weyl}(ii)) 
guarantees that there are infinitely many such pairs.
This completes the proof of the theorem in the case when $1/3<\alpha<1/2$.

Now suppose $1/2<\alpha<1$.
As before, the condition $m+1\in\Bb\cap\Bc$ 
in \eqref{eq:thm5-proof-goal}
is equivalent to the two conditions \eqref{eq:thm5-proof-Bb-condition} and
\eqref{eq:thm5-proof-Bc-condition}.
However, the remaining condition in \eqref{eq:thm5-proof-goal},
``$m\in \Ba$ and $m+2\in\Ba$'', requires a slightly different treatment in the 
case when $\alpha>1/2$.  In this case, 
the number $k$ in  Lemma \ref{lem:beatty-properties}(iv) 
is equal to $k=\fl{1/\alpha}=1$ and we have
$\fp{1/\alpha}\alpha=1-\alpha$. Hence, by Lemma \ref{lem:beatty-properties}(iv) 
a \emph{sufficient} condition for $m\in\Ba$ and $m+2\in\Ba$
to hold is 
\begin{equation}
\label{eq:thm5-proof-Ba-condition2}
u_m<1-\alpha.
\end{equation}
The latter condition is the analog of 
\eqref{eq:thm5-proof-Ba-condition}, which characterizes the integers $m$
satisfying ``$m\in \Ba$ and $m+2\in\Ba$'' in the case when
$1/3<\alpha<1/2$.

The remainder of the argument parallels
that for the case $1/3<\alpha<1/2$.
We seek to show that there exist infinitely many $m$ such
that \eqref{eq:thm5-proof-Bb-condition},
\eqref{eq:thm5-proof-Bc-condition}, and 
\eqref{eq:thm5-proof-Ba-condition2} hold.
Given $\epsilon>0$, set 
\[
I_\epsilon'=(\beta,\beta+\epsilon),\quad 
J_\epsilon=(1-\beta,1-\beta+\epsilon).
\]
If $\epsilon$ is sufficiently small, then, since
$\beta=1-\alpha-\gamma<1-\alpha$, we have 
\begin{equation}
\label{eq:thm5-proof-I-J-conditions1-case2}
I_\epsilon'\subset (0,1-\alpha)
\quad J_\epsilon\subset (1-\beta,1),
\end{equation}
and 
\begin{equation}
\label{eq:thm5-proof-I-J-conditions2-case2}
I_\epsilon'+J_\epsilon\subset (1,2).
\end{equation}
From \eqref{eq:thm5-proof-I-J-conditions1-case2} we immediately see that
\eqref{eq:thm5-proof-Bb-condition} and \eqref{eq:thm5-proof-Ba-condition2}
hold whenever $u_m\in I_\epsilon'$ and $v_m\in J_\epsilon$. 
Moreover, by \eqref{eq:thm5-proof-I-J-conditions2-case2} we have, for
any sufficiently small $\epsilon$,
\begin{equation}
\label{eq:thm5-proof-um+vm-condition-case2}
\fp{u_m+v_m}=u_m+v_m-1<(\beta+\epsilon) 
+(1-\beta +\epsilon)-1=2\epsilon<1-\alpha-\beta,
\end{equation}
so \eqref{eq:thm5-proof-Bc-condition} holds as well whenever
$u_m\in I_\epsilon'$ and $v_m\in J_\epsilon$.
As before, Weyl's Theorem ensures that there are infinitely
many $m$ for which  the latter two conditions hold, and hence infinitely
many $m$ for which \eqref{eq:thm5-proof-um-vm-condition} holds.
This completes the proof of Theorem \ref{thm:thm5}.

\section{Concluding Remarks}
\label{sec:concluding-remarks}

In this section we discuss some related results in the literature, other
approaches to almost Beatty partitions, possible extensions and
generalizations of our results, and directions for future research.

\paragraph{More on iterated Beatty partitions.}
In Theorem \ref{thm:thm2} we constructed almost Beatty partitions by iterating the
standard two-part Beatty partition process.  This is perhaps the most
natural approach to partitions into more than two ``Beatty-like'' 
sequences,  and many special cases of such constructions have appeared in
the literature. Skolem observed in 1957 \cite[p.~68]{skolem1957} that
starting out with the complementary Beatty sequences
$\seq{\fl{n\Phi}}$ and $\seq{\fl{n\Phi^2}}$ and
``partitioning'' the index $n$ in the first sequence into the same pair of
complementary Beatty sequences yields a three part partition consisting of
the sequences $\seq{\fl{\Phi\fl{\Phi n}}}$, $\seq{\fl{\Phi\fl{\Phi^2 n}}}$, and
$\seq{\fl{\Phi^2 n}}$.  Further iterations of this process lead to partitions
into arbitrarily many Beatty-like sequences involving the golden ratio or
other special numbers; see, for example, Fraenkel
\cite{fraenkel1977, fraenkel1994, fraenkel2010}, Kimberling
\cite{kimberling2008}, and Ballot \cite{ballot2017}. 

The ``iterated Beatty partition'' approach leads to sequences whose terms
are given by iterated floor functions. Removing all but the outermost floor
function in such a sequence, one obtains an exact Beatty sequence whose
terms differ from the original sequence by a bounded quantity. Thus, the 
partitions generated by this approach are almost Beatty partitions in the
sense defined of this paper.  However, as our results show, these
partitions do not necessarily yield the smallest perturbation errors.

\paragraph{Partitions into non-homogeneous Beatty sequences.}
Another natural approach to almost Beatty partitions is by considering 
\emph{non-homogeneous} Beatty sequences, that is, sequences of the form
$\seq{\fl{(n+\beta)/\alpha}}$.  Clearly, any sequence of this form differs
from the exact Beatty sequence $\seq{\fl{n/\alpha}}$ by a bounded amount,
and hence is an almost Beatty sequence. Thus, any partition of $\NN$ into
non-homogeneous Beatty sequences is a partition into almost Beatty
sequences.  The question of when a partition of $\NN$ into  $k$
non-homogeneous Beatty sequences exists has received considerable attention
in the literature, but a  complete solution is only known in
the case $k=2$;  see, for example,  Fraenkel \cite{fraenkel1969} and
O'Bryant \cite{obryant2003} for the case $k=2$, and Tijdeman
\cite{tijdeman2000} and the references therein for the general case.
An intriguing question is whether any of the sequences constructed 
in Theorems \ref{thm:thm1}--\ref{thm:thm4} can be represented as a
non-homogeneous Beatty sequence.

Very recently, Allouche and Dekking \cite{allouche-dekking2018}
considered two- and three-part partitions of $\NN$ into sequences of the
form $\seq{p\fl{\alpha n}+qn+r}$, where $p,q,r$ are integers and
$\alpha$ is  an irrational number.  Such sequences can be viewed as
generalized non-homogeneous Beatty sequences, and they represent almost Beatty
sequences in the sense of this paper, associated with the exact Beatty
sequence $\seq{\fl{(p\alpha+q)n}}$.  Again, it would be interesting to
see if any of the sequences constructed in our partitions are of the 
above form. 

\paragraph{Other approximation measures.}
In our results we measured the ``closeness'' of two sequences $A=\seq{a(n)}$
and $B=\seq{b(n)}$  by the sup-norm $\|a-b\|$.  This norm has the natural
interpretation as the maximal amount by which an element in one sequence
needs to be ``perturbed'' in order to obtain the corresponding element in
the other sequence.  An alternative, and seemingly equally natural,
measure of closeness of two sequences $A$ and $B$ is the sup-norm of the
associated counting functions, i.e., 
\[
\|A-B\|=\sup\{|A(n)-B(n)|: n\in\NN\}.
\]
One can ask how our results would be affected if the
approximation errors had been measured in terms of the latter norm.
Surprisingly, this question has a trivial answer, at least in the case of
Theorems \ref{thm:thm1}, \ref{thm:thm3}, and \ref{thm:thm4}:  In all of
these cases, we have $\|\Bb-\Bbt\|\le 1$ and $\|\Bc-\Bct\|\le 1$ whenever
the conditions of the theorems are satisfied. In other words, when measured
by the counting function norm, the almost Beatty sequences constructed in
these theorems are either exact Beatty sequences or 
just one step away from being an exact Beatty sequence.  This follows easily from Lemmas \ref{lem:Ba-Bb-Bc-counting-function}, 
\ref{lem:Ebt-formula}, and \ref{lem:Ect-formula}.
The reason for the simple form of this result  is that the
counting function norm is a much less sensitive measure than the
element-wise norm we have used in this paper.

\paragraph{Partitions into $n$ almost Beatty sequences.}
A natural question is whether the constructions of Theorems
\ref{thm:thm2}--\ref{thm:thm4} can be generalized to yield partitions of
$\NN$ into more than three almost Beatty sequences.  The ``iterated Beatty
partition'' approach in Theorem \ref{thm:thm2} lends itself easily to such a
generalization, though it is not clear what the best-possible perturbation
errors are that can be achieved in the process. 

\paragraph{Densities of perturbation errors.}
Another question concerns the densities with which the perturbation errors
in our results occur.  We have computed these densities in the case of
Theorem \ref{thm:thm3} under an appropriate linear independence assumption
(see the end of Section \ref{sec:proof-thm3}), but in the case of Theorem
\ref{thm:thm1} this would be a much more difficult undertaking 
(see Remark \ref{rem:error-densities-thm1}). One motivation for computing 
these densities is that we can then consider a refined approximation
measure given by the weighted average of the absolute values of the errors,
with the weights being the associated densities.  It would be interesting
to see which approaches yield the closest approximation to a partition into
Beatty sequences in terms of this refined measure of closeness.

\paragraph{Connection with optimal assignment problems.}
Finally we remark on an interesting connection with a class of problems
involving the assignment or scheduling of tasks in a manner that is, in an
appropriate sense,  most efficient or most fair. Problems of this type
arise in many contexts, from distributing tasks among different machines,
to allocating seats in an election, to assigning drivers in a carpool;
see, for example, Coppersmith et al. \cite{coppersmith2011}. 

We describe one such problem, the  {\em chairman assignment problem} of
Tijdeman \cite{tijdeman1980}: Suppose $k$ states with weights
$\lambda_1,\dots,\lambda_k$, where the $\lambda_i$ are positive real
numbers satisfying $\sum_{i=1}^k\lambda_i=1$, form a union.  Each year one
of the states is chosen to designate the chair of the union.  The goal is
to perform this assignment in such way that, at any time, the proportion
of chairs chosen from a given state up to that time is as close as
possible to the weight of that state.  In his paper, Tijdeman determines
an assignment that is optimal in a certain global sense, and he gives a
recursive algorithm to compute the optimal assignment sequence.  

Mathematically, an assignment of chairs amounts to a partition of the
positive integers into $k$ sets.  Partitions into $k$ almost Beatty sequences
with densities given by the weights $\lambda_i$ of the states  provide natural
candidates for a ``fair'' assignment. In particular, any of the partitions
produced by Theorems \ref{thm:thm1}--\ref{thm:thm4} yields a possible chair
assignment sequence for the case of three states with weights $\alpha$, $\beta$,
and $\gamma$. These partitions turn out to be different from those given 
by Tijdeman, and they are not optimal with respect to the measure used 
by Tijdeman (which essentially amounts to minimizing the discrepancies
$\sup_n|A_i(n)-\lambda_in|$, where $A_i$ is the $i$th set in the
partition).  However, as shown in this paper, these partitions are, in
many cases, best-possible in a different sense and thus might be of
interest in applications to assignment problems.  Exploring these
connections and applications further could be a fruitful area for future
research.

\section{Acknowledgments}

This work originated with an undergraduate research project carried out in
Spring 2018 at the \emph{Illinois Geometry Lab} at the University of
Illinois; we thank the IGL for providing this opportunity.  We also thank
Michel Dekking for calling our attention to the paper 
\cite{allouche-dekking2018} and the editor-in-chief for referring us 
to Tijdeman's paper \cite{tijdeman1980} on the chairman assignment problem.

\bigskip
\hrule
\bigskip

\noindent 2010 {\it Mathematics Subject Classification}:
Primary 11B83; Secondary 05A17, 11P81.

\noindent \emph{Keywords: } 
Beatty sequence, complementary sequence, partition of the integers.

\bigskip
\hrule
\bigskip

\noindent (Concerned with sequences
\seqnum{A000201},
\seqnum{A003144},
\seqnum{A003145},
\seqnum{A003146},
\seqnum{A003623},
\seqnum{A004919},
\seqnum{A004976},
\seqnum{A158919},
\seqnum{A277722},
\seqnum{A277723}, 
and
\seqnum{A277728}.)

\bigskip
\hrule
\bigskip

\vspace*{+.1in}
\noindent
Received October 5 2018; 
revised version received  July 3 2019.
Published in {\it Journal of Integer Sequences}, July 7 2019.

\bigskip
\hrule
\bigskip

\noindent
Return to
\htmladdnormallink{Journal of Integer Sequences home
page}{http://www.cs.uwaterloo.ca/journals/JIS/}.
\vskip .1in


\begin{thebibliography}{10}

\bibitem{allouche-dekking2018}
J.~P.~Allouche and F.~M.~Dekking,  Generalized Beatty sequences
and complementary triples, 
preprint, 2018. Available at \url{https://arXiv.org/abs/1809.03424}.

\bibitem{ballot2017}
C.~Ballot, On functions expressible as words on a pair of {B}eatty
  sequences, {\em J.~Integer Sequences} {\bf 20} (2017), 
\href{https://cs.uwaterloo.ca/journals/JIS/VOL20/Ballot/ballot22.html}
  {Article 17.4.2}.

\bibitem{bang1957}
T.~Bang, On the sequence {$[n\alpha],\;n=1,2,\cdots$}. {S}upplementary
  note to the preceding paper by {T}h. {S}kolem, {\em Math. Scand.} {\bf 5}
  (1957), 69--76.

\bibitem{beatty}
S.~Beatty, Problem 3173, {\em Amer. Math. Monthly} {\bf 33} (1926), 159.

\bibitem{putnam1959}
L.~E.~Bush, The {W}illiam {L}owell {P}utnam {M}athematical {C}ompetition, {\em
  Amer. Math. Monthly} {\bf 68} (1961), 18--33.


\bibitem{coppersmith2011}
D.~Coppersmith, T.~J.~Nowicki, G.~A.~Paleologo, C.~Tresser, and C.~W.~Wu,
The optimality of the online greedy algorithm in carpool and chairman
assignment problems, 
{\em ACM Trans. Algorithms} \textbf{7} (2011), Art. 37.


\bibitem{fraenkel1969}
A.~S.~Fraenkel, The bracket function and complementary sets of integers,
  {\em Canad. J.~Math.} {\bf 21} (1969), 6--27.

\bibitem{fraenkel1977}
A.~S.~Fraenkel, Complementary systems of integers, {\em Amer. Math.
  Monthly} {\bf 84} (1977), 114--115.

\bibitem{fraenkel1994}
A.~S.~Fraenkel, Iterated floor function, algebraic numbers, discrete
  chaos, {B}eatty subsequences, semigroups, {\em Trans. Amer. Math. Soc.} {\bf
  341} (1994), 639--664.

\bibitem{fraenkel2010}
A.~S.~Fraenkel, Complementary iterated floor words and the {F}lora game,
  {\em SIAM J.~Discrete Math.} {\bf 24} (2010), 570--588.

\bibitem{ginosar-yona2012}
Y.~Ginosar and I.~Yona, A model for pairs of {B}eatty sequences, {\em
  Amer. Math. Monthly} {\bf 119} (2012), 636--645.

\bibitem{graham1963}
R.~L.~Graham, On a theorem of {U}spensky, {\em Amer. Math. Monthly} {\bf 70}
  (1963), 407--409.

\bibitem{holshouser-reiter}
A.~Holshouser and H.~Reiter, A generalization of {B}eatty's theorem,
  {\em Southwest J.~Pure Appl. Math.}  (2001), 24--29.

\bibitem{kedlaya-putnam}
K.~S.~Kedlaya, B.~Poonen, and R.~Vakil, {\em The {W}illiam {L}owell
  {P}utnam {M}athematical {C}ompetition, 1985--2000},
  Mathematical Association of America, 2002.

\bibitem{kimberling2008}
C.~Kimberling, Complementary equations and {W}ythoff sequences, {\em J.
  Integer Sequences} {\bf 11} (2008),
\href{https://cs.uwaterloo.ca/journals/JIS/VOL11/Kimberling/kimberling719a.html}
  {Article 08.3.3}.



\bibitem{kimberling-stolarsky2016}
C.~Kimberling and K.~B.~Stolarsky, Slow {B}eatty sequences, devious
  convergence, and partitional divergence, {\em Amer. Math. Monthly} {\bf
  123} (2016), 267--273.

\bibitem{putnam1995}
L.~F.~Klosinski, G.~L.~Alexanderson, and L.~C.~Larson, The
  {F}ifty-{S}ixth {W}illiam {L}owell {P}utnam {M}athematical {C}ompetition,
  {\em Amer. Math. Monthly} {\bf 103} (1996), 665--677.

\bibitem{kuipers-niederreiter}
L.~Kuipers and H.~Niederreiter, {\em Uniform Distribution of Sequences},
  Wiley-Interscience,  1974.

\bibitem{niven-book}
I.~Niven, {\em Diophantine Approximations}, Dover Publications, 2008.

\bibitem{obryant2002}
K.~O'Bryant, A generating function technique for {B}eatty sequences and
  other step sequences, {\em J.~Number Theory} {\bf 94} (2002), 299--319.

\bibitem{obryant2003}
K.~O'Bryant, Fraenkel's partition and {B}rown's decomposition, {\em
  Integers} {\bf 3} (2003), A11.

\bibitem{rayleigh-book}
Lord Rayleigh, {\em The Theory of Sound}, Macmillan, 1894.

\bibitem{schoenberg-book}
I.~J.~Schoenberg, {\em Mathematical Time Exposures}, Mathematical
  Association of America, 1982.

\bibitem{skolem1957}
T.~Skolem, On certain distributions of integers in pairs with given
  differences, {\em Math. Scand.} {\bf 5} (1957), 57--68.

\bibitem{oeis}
N.~J.~A.~Sloane, \emph{The On-Line Encyclopedia of Integer Sequences},
2018. Published electronically at \url{https://oeis.org}.

\bibitem{stolarsky1976}
K.~B.~Stolarsky, Beatty sequences, continued fractions, and certain shift
  operators, {\em Canad. Math. Bull.} {\bf 19} (1976), 473--482.


\bibitem{tijdeman1980}
R.~Tijdeman, 
The chairman assignment problem,
{\em Discrete Math.}
{\bf 32} (1980), 323--330.


\bibitem{tijdeman2000}
R.~Tijdeman, 
Exact covers of balanced sequences and {F}raenkel's conjecture,
in {\em Algebraic Number Theory and Diophantine Analysis ({G}raz,
  1998)}, de Gruyter,  2000, pp.~467--483. 

\bibitem{uspensky1927}
J.~V.~Uspensky, On a problem arising out of the theory of a certain game, {\em
  Amer. Math. Monthly} {\bf 34} (1927), 516--521.

\end{thebibliography}
\end{document}